\documentclass{amsart}

\usepackage{amsmath,amssymb}
\usepackage{float}
\usepackage{color}
\usepackage{graphicx}

\usepackage{enumerate}
\numberwithin{equation}{section}

\usepackage{hyperref}
\usepackage{verbatim}

\newcommand{\var}{\operatorname{var}}
\newcommand{\Int}{\operatorname{int}}
\newcommand{\esup}{\operatorname{ess\,sup}}

\newcommand{\elim}{\operatorname{ess\,lim}}

\newcommand{\Reg}{\operatorname{Reg}}
\newcommand{\reg}{\operatorname{reg}}
\newcommand{\fr}{\operatorname{frac}}

\def\j{R}
\def\T{\mathcal{T}}
\def\eq{equation}

\def\ep{\varepsilon}

\def\Si{\Sigma}
\def\s{\sigma}

\def\Sn{\Sigma_n}
\def\St{\Sigma_*}
\def\Sb{\Sigma_b}

\def\r{\rho}

\def\<{\langle}
\def\>{\rangle}

\def\E{\mathbb{E}}
\def\N{\mathbb{N}}
\def\P{\mathbb{P}}
\def\R{\mathbb{R}}
\def\rd{\mathbb{R}^d}

\def\K{\mathcal{K}}

\def\F{\mathcal{F}}
\def\G{\mathcal{G}}
\def\Q{\mathcal{Q}}

\def\T{\mathcal{T}}
\def\bT{\mathbb{T}}

\def\ft{\mathfrak{t}}
\def\bff{\mathbf{f}}
\def\bfr{\mathbf{r}}

\def\1{\mathbf{1}}

\newcommand{\cl}{\operatorname{cl}}
\newcommand{\INt}{\operatorname{int}}
\newcommand{\conv}{\operatorname{conv}}
\newcommand{\reach}{\operatorname{reach}}
\newcommand{\Tan}{\operatorname{Tan}}

\newcommand{\cL}{\mathcal{L}}
\newcommand{\crit}{\operatorname{crit}}
\newcommand{\ed}{\stackrel{d}{=}}

\theoremstyle{plain}
   \newtheorem{thm}{Theorem}[section]
   \newtheorem{thms}{Theorem}[section]

  \newtheorem{lems}[thms]{Lemma}
   \newtheorem{lem}[thm]{Lemma}
   \newtheorem{props}[thms]{Proposition}

\theoremstyle{remark}

\newtheorem*{remark*}{Remark}

\theoremstyle{definition}
   
   \newtheorem{defs}[thms]{Definition}
   
   \newtheorem{rem}[thm]{Remark}
   \newtheorem{rems}[thms]{Remark}
   
\newtheorem{ex}[thms]{Example}

\begin{document}

\title[Mean fractal curvatures of code tree fractals]{Mean Minkowski content and mean fractal curvatures of random self-similar code tree fractals
}

\author{Jan Rataj}\curraddr{\it Jan Rataj, Charles University, Faculty of Mathematics and Physics, Soko\-lovs\-k\'a 83, 186 75 Praha 8, Czech Republic}
\author{Steffen Winter}\curraddr{\it Steffen Winter, Institute of Stochastics, Karlsruhe Institute of Technology, Englerstr. 2, D-76131 Karlsruhe, Germany}
\author{Martina Z\"ahle}\curraddr{\it Martina Z\"ahle, Institute of Mathematics,  University of Jena, Ernst-Abbe-Platz 2, D-07743 Jena, Germany.}

\begin{abstract}
We consider a class of random self-similar fractals based on code trees which includes random recursive, homogeneous and V-variable fractals and many more. For such random fractals we consider mean values of the Lipschitz-Killing curvatures of their parallel sets for small parallel radii. Under the uniform strong open set condition and some further geometric assumptions we show that rescaled limits of these mean values exist as the parallel radius tends to $0$. Moreover, integral representations are derived for these limits which recover and extend those known in the deterministic case and certain random cases. Results on the mean Minkowski content are included as a special case and shown to hold under weaker geometric assumptions.
\end{abstract}

\thanks{Steffen Winter and Martina
Zähle received support from the DFG grant 433621248.}



\maketitle

\section{Introduction}\label{introduction}
For a compact set $F\subset\R^d$ and $k\in\{0,\ldots,d-1\}$, the (total) \emph{fractal Lipschitz-Killing curvatures} of $F$ of order $k$ are introduced as rescaled (essential or averaged) limits
\begin{align} \label{eq:def-frac-curv}
    \elim_{\ep\rightarrow 0}\limits\ep^{D-k}\,C_k(F(\ep)) \qquad \text{ or }\qquad \lim_{\delta\rightarrow 0}\frac{1}{|\ln\delta|}\int_\delta^1\ep^{D-k}\, C_k(F(\ep))~\ep^{-1}d\ep,
\end{align}
provided these limits exist. Here, for $\ep>0$, $F(\ep)$ denotes the closed $\ep$-neighborhood of $F$ (see \eqref{eq:def-parset}), $C_k(A)$ is the total mass of the $k$-th curvature measure $C_k(A,\cdot)$ of a compact set $A$, and $D$ a suitable scaling exponent.

$C_{d-1}(A)$ corresponds to the surface area of the boundary $\partial A$ of $A$ as a marginal case.
If $\partial A$ is $C^2$-smooth, then, up to some normalizing constants, the total curvature $C_k(A)$ agrees with the integral over the boundary of the symmetric function of order $d-1-k$ of the principal curvatures. These are well-known geometric invariants in differential geometry. In geometric measure theory, extensions to more general sets are known, for which symmetric functions of generalized principal curvatures are integrated over the unit normal bundles of the sets.
(The Lipschitz-Killing differential forms provide an equivalent approach.)
Here we do not need the explicit representation of these curvatures and the corresponding  measures, but only some essential properties, which are recalled in Section \ref{Preliminaries}.

For completeness we also include the volume $C_d(F(\ep)):=\mathcal{L}^d(F(\ep))$ into our considerations. For $k=d$, the limits in \eqref{eq:def-frac-curv} are well known as Minkowski content and average Minkowski content. In this case, the exponent $D$ in \eqref{eq:def-frac-curv} is clearly the Minkowski dimension. Also for the other indices $k$ the scaling exponent $D$ is typically the Minkowski dimension of $F$, provided $F$ is a sufficiently nice fractal.

If $F$ is a random fractal, then one may ask for the existence of the limits in \eqref{eq:def-frac-curv} in the almost sure sense. While such almost sure limits exist in some cases, e.g. for some random recursive constructions, see \cite{Ga00} for
$k=d$ and \cite{Za11} for all $k$, this cannot be expected in general. One can then still ask for the existence of corresponding limits of expectations:
\begin{align} \label{eq:def-mean-frac-curv}
    \elim_{\ep\rightarrow 0}\limits\ep^{D-k}\,\E C_k(F(\ep)) \qquad \text{ or }\qquad \lim_{\delta\rightarrow 0}\frac{1}{|\ln\delta|}\int_\delta^1\ep^{D-k}\, \E C_k(F(\ep))~\ep^{-1}d\ep.
\end{align}
In the present paper it is our aim to study such limits of expectations for a large class of random self-similar sets, which allow dependencies, and establish conditions guaranteeing their existence.
The model we introduce will include as special cases the classical random recursive sets of \cite{Fa86,Gr87,MW86}, homogeneous random fractals as introduced in \cite{Ha92}, \cite{Fa95}, and also $V$-variable fractals as defined in \cite{BHS12} and extended in \cite{Za23}.
For those classes of sets some mean value results (i.e., results regarding the limits in \eqref{eq:def-mean-frac-curv}) are known, see \cite{Ga00}, \cite{Za11} for the case of random recursive sets, \cite{Za20}, \cite{RWZ23} for homogeneous fractals, and \cite{Za23} for $V$-variable fractals.
In the latter case such mean value results are only known for $k=d$ and $d-1$, hence our results are also novel for V-variable fractals. The integral formulas we obtain for the limits in \eqref{eq:def-mean-frac-curv} are structurally the same as the formulas known in these special cases. For the random recursive case mentioned above the a.s.\ limits are equal to those of the expectations. This is certainly not true in general. For example, in \cite{Tr23} it is shown under some mild conditions on homogeneous random fractals that the a.s.\ average Minkowski content does not exist. Hence, while in some cases almost sure limits as in \eqref{eq:def-frac-curv} do not exist, one still has existence of the limits in the mean in \eqref{eq:def-mean-frac-curv}.

The random fractals we introduce are constructed by means of random labeled code trees, in which each vertex is labeled by a random iterated function system (RIFS), i.e., a random family of contracting similarities. In contrast to previous models we allow for more dependencies between different paths in the tree as well as between the RIFSs in the same construction step.

In Section \ref{Preliminaries} some notions and properties related to curvature measures are recalled, in particular, geometric conditions for their existence.
The random code trees with RIFSs as labels are introduced in Section \ref{sect_code_tree_construction}. Here the back path property in Assumption (A2) is essential for our purposes. It is fulfilled for all former models mentioned above.
Some examples are presented in Section \ref{Examples}. For the random Sierpi\'nski gaskets discussed in Example~\ref{ex:dependent-SG}, the families of RIFSs associated to different construction steps form an independent sequence. However, within each construction step the RIFSs are dependent.
Example \ref{ex:3} provides a simple model, where dependencies between different levels of the labeled tree occur.
The random carpets in Example~\ref{ex:percolation} are defined in such a way that neighboring squares along horizontal lines are chosen dependent on each other. This model also exhibits dependencies between the levels.
The main results, regarding the existence of the limits in \eqref{eq:def-mean-frac-curv} for our new model, are formulated in Section \ref{sec:main_res}, see Theorem \ref{maintheorem}, Remark \ref{rem:partial}, and Theorem \ref{positive}. In particular, it follows that the a.s.\ Minkowski dimension of such random fractals can be strictly less than the Minkowski dimension in the mean sense, cf. Remark \ref{rem:dimensions}.

Section \ref{Proofs} contains the proofs. In Remark \ref{rem:extension} a possible extension to a more general class of models is roughly described, which includes the model of random self-similar code trees with necks from \cite{Tr21}.

\section{Preliminaries}\label{Preliminaries}

\subsection*{Curvature measures}
Our basic framework for curvature measures are \emph{sets with positive reach} introduced in the seminal work of Federer \cite{Fe59}. The \emph{reach} of a nonempty closed subset $A\subset\R^d$, denoted $\reach A$, is the supremum of all $r\geq 0$ such that any point $x\in\R^d$ with distance from $A$ less than $r$ has a unique closest point in $A$ (denoted by $\pi_A(x)$). If $\reach A>0$, then the \emph{curvature measures} $C_k(A,\cdot)$ of $A$, $k=0,\dots,d$, are defined as signed Radon measures in $\R^d$, with support in $\partial A$ if $k\leq d-1$; $C_d(A,\cdot)$ is the restriction of the Lebesgue measure $\cL^d$ to $A$. We refer to the original definition in \cite{Fe59} and to the survey \cite[Section~4]{RZ19}. The following selected properties of curvature measures, valid for any $k\in\{0,\dots,d\}$, $A\subset\R^d$ with positive reach and $B\subset\R^d$ Borel, will be needed here.

\begin{description}
    \item[motion covariance]
    $C_k(g(A),g(B))=C_k(A,B)$ for any Euclidean motion $g:\R^d\to\R^d$.
    \item[homogeneity]
    $C_k(rA,rB)=r^k C_k(A,B)$, $r>0$.
    \item[locality] If $U$ is open and $A'\subset\R^d$ is another set with positive reach such that $A\cap U=A'\cap U$, then
    $C_k(A,B)=C_k(A',B)$ for any $B\subset U$.
\end{description}
If a set $A\subsetneq\R^d$ with positive reach is \emph{full-dimensional}, which means that the tangent cone $\Tan(A,a)$ has dimension $d$ at any $a\in A$, then the curvature measures can be defined consistently also for the closure of the complement of $A$ in the way that
$$C_k(\cl(\R^d\setminus A),\cdot)=(-1)^{d-1-k}C_k(A,\cdot),\quad k=0,\dots,d-1,$$
see \cite[Example~9.10]{RZ19}.

\subsection*{Regularity of the distance function}
We use the notation $d_A$ for the distance function of a closed set $A\subset\R^d$, i.e., $d_A(x):=\inf_{a\in A}\|x-a\|$.
$d_A$ is a Lipschitz function and we say that $x\in\R^d$ is a \emph{critical point} of $d_A$ if the Clarke subgradient of $d_A$ at $x$, denoted $\partial d_A(x)$, contains the origin. An equivalent geometric condition for $x$ being a critical point of $d_A$ is that $x$ lies in the convex hull of the set of closest points to $x$ in $A$. We refer to \cite{Fu85} for the properties of critical points and critical values.
A number $r>0$ is called a \emph{critical value} of $d_A$ if $r=d_A(x)$ for some critical point $x$, and it is called \emph{regular} otherwise. We will write $\crit A$ for the set of all critical values of $d_A$. The set $\crit A$ is always closed in $(0,\infty)$, and it is a Lebesgue null set in space dimensions $d=2$ and $3$.

Recall that
\begin{align} \label{eq:def-parset}
    A(r):=\{x\in\R^d:\, d_A(x)\leq r\}
\end{align}
denotes the closed $r$-neighbourhood of $A$, $r>0$.
In \cite{Fu85}, Fu showed that if $r\not\in\crit A$ then $\partial(A(r))$ is a Lipschitz surface and $\reach \cl(\R^d\setminus A(r))>0$. If $d=2$ or $3$, then in view of the above mentioned property of critical values, this is true for Lebesgue almost all $r>0$. Moreover, in any dimension, if $r$ is large enough compared to the diameter $|A|$ of $A$, then $r$ is a regular value of $d_A$ and there is even a quantitative lower estimate for the reach of the closed complement $\cl(\R^d\setminus A(r))$ of $A(r)$ as well as

an upper bound for the total variation measure $C_k^{\var}(A(r),\cdot)$ of $C_k(A,\cdot)$:

\begin{lems}[{\cite[Theorem~4.1]{Za11}}] \label{lems:large distances}

For any $R>\sqrt{2}$ and $k=0,1,\ldots,d$ there exists a constant $c_k(R)$ such that for any compact set $A\subset\mathbb{R}^d$ and any $r\ge R|A|$,
$$\reach(\cl(\R^d\setminus A(r)))\ge|A|\sqrt{R^2-1},$$
$\partial (A(r))$ is a $(d-1)$-dimensional Lipschitz submanifold, and
$$\sup_{r\ge R|A|}\frac{C_k^{\var}(A(r),\rd)}{r^k}\le c_k(R)\, .$$
\end{lems}

We will also need the following result on weak continuity of curvature measures of parallel sets. Recall that the set of regular values of $d_A$ is open in $(0,\infty)$.

\begin{lems}[{\cite[Theorem~6.1]{RWZ23}}]  \label{L_cont}
If $A\subset\R^d$ is nonempty and compact, $r_0>0$ is a regular value of $d_A$ and $k\in\{0,\dots,d\}$ then $C_k(A(r),\cdot)$ converges weakly to $C_k(A(r_0),\cdot)$ as $r\to r_0$.
\end{lems}

\begin{rems}
    If $k=d$, then the assertion of Lemma~\ref{L_cont} holds for any $r_0>0$ (i.e., $r_0$ does not need to be a regular value of the distance function), and if $k=d-1$ then it holds for all $r_0>0$ up to a countable set, see \cite[Remark~2.5]{RWZ23}.
\end{rems}

\subsection*{Random closed sets and measurability issues}
Random closed sets in $\R^d$ are random variables taking values in the space of closed subsets of $\R^d$ equipped with the Fell topology. On the subfamily of nonempty compact sets (denoted by $\K$), the Fell topology agrees with that induced by the Hausdorff distance. We briefly recall some measurability properties, and refer to \cite{RWZ23} for further details.

The sets with positive reach form a measurable subfamily of the family of closed sets. The mappings $(r,K)\mapsto K(r)$ and $(r,K)\mapsto\cl(\R^d\setminus K(r))$
are measurable on $(0,\infty)\times\K$. Furthermore, for any $k\in\{0,\dots d\}$ and any bounded Borel set $B\subset\R^d$, the mappings $K\mapsto C_k(K,B)$ and $K\mapsto C_k^{\var}(K,B)$ are measurable functions on the family of sets with positive reach.
Denote by
$$\Reg:=\{(r,K)\in (0,\infty)\times\K:\, r\not\in\crit K\}$$
the set of \emph{regular pairs}, i.e., of pairs $(r,K)$ such that $r$ is a regular value of the distance function $d_K$. Note that the definitions of $\Reg$ given in \cite{Za11} and \cite{RWZ23} are equivalent, see \cite[Section~6]{RWZ23}.
$\Reg$ is a Borel subset of $(0,\infty)\times\K$.

\section{Construction of random self-similar code tree fractals}\label{sect_code_tree_construction}

We will introduce a class of random self-similar fractals

based on a labeled code tree construction extending several models in the literature.

As in \cite{RWZ23}, we start with a \emph{random iterated function system (RIFS)} $\F=(f_1,\dots,f_N)$ consisting of a random number $N$ of contracting similarities
$f_1,\ldots,f_N$ with (random) contraction ratios $r_1,\ldots,r_N$ satisfying $r_{\min}\leq r_i\leq r_{\max}$ for some deterministic values $0<r_{\min}\leq r_{\max}<1$.
Formally, $\F$ is a random variable with values in the space $(\Omega_0,\mathfrak{F}_0)$ of finite sequences of contracting similarities with contraction ratios in $[r_{\min},r_{\max}]$, and with the natural $\sigma$-algebra $\mathfrak{F}_0$ induced by uniform convergence of continuous functions on compact sets.
We assume throughout that $1<\E N<\infty$ and that the \emph{Uniform Open Set Condition} (UOSC) holds: there exists a (deterministic) nonempty open set $O$ such that almost surely,
\begin{\eq}\label{UOSC}
f_i(O)\subset O~{\rm and}~f_i(O)\cap f_j(O)=\emptyset\,,~i\neq j.
\end{\eq}

Let $\N^{<\infty}:=\bigcup_{n=0}^\infty \N^n$ be the space of finite sequences of natural numbers. We use the usual short notation $\s_1\s_2\ldots\s_n:=(\s_1,\s_2,\ldots,\s_n)$ for its elements. We will consider an infinite \emph{random tree} $\Sigma_*\subset\N^{<\infty}$, containing the empty sequence $\emptyset$ as its root, where the tree structure is understood with respect to the natural ordering in $\N^{<\infty}$, i.e., two sequences $\tau,\s\in\N^{<\infty}$ form an edge iff there is some $i\in\N$ such that $\s=\tau i$ or $\tau=\s i$.  Each node $\sigma\in\Sigma_*$ will be equipped with a label $\F_{\sigma}=(f_{\s 1},\ldots,f_{\s{N_\s}})$, which is an RIFS.
For any $n\in\N_0$, we denote by $\Sn$ the subset of $\Sigma_*$ of all sequences of length $n$, and for $\s=\sigma_1\sigma_2\ldots\sigma_n\in\Sn$, we set
\begin{\eq}\label{eq:bfs}
\bff_{\sigma}:=f_{\sigma_1}\circ f_{\sigma_1\sigma_2}\circ\dots\circ f_{\sigma_1\ldots\sigma_n}\,.
\end{\eq}
Observe that $\bff_\sigma$ is a similarity mapping with contraction ratio
$$\bfr_\s:= r_{\sigma_1}\cdot r_{\sigma_1\sigma_2}\cdot\ldots\cdot r_{\sigma_1\ldots\sigma_n}\,.$$
The random tree $\Sigma_*$ together with the associated (random) labels form the \emph{random labeled tree} $\T:=\{(\s,\F_{\sigma}): \s\in\St\}$.

Let us describe now in detail how $\T$ is constructed. Consider the space $\mathbb{T}:=(\{\emptyset\}\cup\Omega_0)^{(\N^{<\infty})}$ of all mappings $\ft:\N^{<\infty}\to \{\emptyset\}\cup\Omega_0$ and equip it with the usual $\sigma$-algebra $\mathfrak{T}$ induced by finite-dimensional cylinders. We will work throughout in the probability space $(\mathbb{T},\mathfrak{T},\P)$, where $\P$ is some probability measure on $\mathbb{T}$ (satisfying some additional assumptions which we specify later). We define the random variables $\Sigma_*$ and $\T$ in $(\mathbb{T},\mathfrak{T},\P)$ by setting, for each $\ft\in\mathbb{T}$,
$$
\Sigma_*(\ft):=\{\s\in\N^{<\infty}: \ft(\sigma)\neq\emptyset\}
$$

and
$$
\T(\ft):=\{(\s,\ft(\s)): \s\in\St\}.
$$
Note that the set $\T(\ft)$ is the graph of $\ft$ restricted to $\Sigma_*(\ft)$ and that any $\ft\in\mathbb{T}$ can be reconstructed from $\T(\ft)$. (Indeed, $\ft(\sigma)=\emptyset$ for all $\sigma\not\in\Sigma_*(\ft)$.) Hence there is a one-to-one correspondence between $\ft$ and $\T(\ft)$, and thus $\P$ may be interpreted as the distribution of $\T$.
The assumptions (A1) and (A2) below on the distribution $\P$ will ensure that  $\Sigma_*$ is a tree $\P$-a.s., justifying to call $\T$ a (random) labeled tree.

For any $\sigma\in\N^{<\infty}$, we introduce the \emph{shift by} $\sigma$ on $\bT$, defined by
$$\theta_\sigma: \ft\mapsto\ft\circ\iota_\sigma,\quad \ft\in\bT,$$
where
$$\iota_\sigma(\tau):=\sigma\tau,\quad \tau\in\N^{<\infty}$$
and $\s\tau$ denotes the concatenation of the sequences $\s$ and $\tau$.
If $X$ is any random variable on the basic probability space $\bT$ and $\sigma\in\N^{<\infty}$, we denote by $X^{[\sigma]}$ the shifted random variable
$$X^{[\sigma]}:\ft\mapsto X(\theta_\sigma\ft).$$
As a basic example of such a shifted variable, we introduce for the labeled tree $\T$ and any $\sigma\in \Sigma_*$, the labeled tree
$$ \T^{[\s]}=\{(\tau,\F_{\s\tau}): \s\tau\in\St\}\,,$$
which is a subtree of $\T$ rooted at vertex $\sigma$. Note that the RIFS $\F_{\s}$ associated to a node $\s\in\Sigma_*$ may also be viewed as a shift (of the RIFS $\F_\emptyset$, which we identify in the sequel with the primary RIFS $\F$ for convenience), that is, $\F_\s=\F^{[\s]}$. Similarly, we have $N_\s=N^{[\s]}$, i.e., the number of mappings in $\F_{\s}$ is a shift of the number $N=N_\emptyset$ of mappings in $\F$. We also set $N_{\s}:=0$ if $\F_{\s}=\emptyset$.

We impose the following two assumptions on the distribution $\P$.
\begin{enumerate}
\item[(A1)] \label{A1} $\P$-a.s.\ $\ft(\emptyset)\neq\emptyset$ and, for any $\s\in\N^{<\infty}$ and $i\in\N$, $\ft(\s i)=\emptyset$, if $N_{\s}<i$, and $\ft(\s i)\neq\emptyset$ otherwise.
(Note that this implies
$\Sigma_0=\{\emptyset\}$ and for all $n\geq 1$,

$$\Sn=\{\s_1\ldots\s_n: \s_1\ldots\s_{n-1}\in\Si_{n-1}, 1\leq\s_n\leq N_{\s_1\ldots\s_{n-1}}\}.)$$

\item[(A2)] \label{A2} For any $i\in \N$ with $\P(i\leq N)>0$, under the condition  $i\leq N$, the random labeled subtree $\T^{[i]}$ is independent of $f_i$ and has the same distribution as $\T$. More formally, this means that, for all measurable subsets $C\subset\Omega_0$ and $B\subset\bT$,
\begin{align*}\label{A_2}
&\P[f_i\in C,\,\T^{[i]}\in B \mid i\leq N]=\P[f_i\in C \mid i\leq N]\,\P[\T^{[i]}\in B \mid i\leq N]\\
&\text{ and } \quad \P[\T^{[i]}\in B\mid i\leq N]=\P(\T\in B).
\end{align*}

\end{enumerate}

\begin{defs} \label{def:codetree}
The labeled random tree
\begin{\eq}\label{labeled code tree}
\T:=\{(\s,\F_{\s}): \s\in\St\}
\end{\eq}
fulfilling (A1) and (A2), where $\F_\emptyset$ has the same distribution as $\F$,  will be called \emph{random self-similar code tree with back path independence}.
\end{defs}

\begin{rems}
    \begin{enumerate}
        \item[(i)]
    Note that for different indices $i$, the subtrees $\T^{[i]}$ may or may not depend on each other.
        \item[(ii)]
From iterated application of (A2) we infer that, for all $n\in\N$ and all $\sigma=\sigma_1\dots\sigma_n\in\N^n$ with $\P(\sigma\in\Sigma_n)>0$,
under the condition  $\sigma\in\Sigma_n$, the random labeled subtree $\T^{[\sigma]}$ is independent of the mappings $f_{\s_1},f_{\s_1\s_2},\dots,f_{\s}$ and has the same distribution as $\T$.
In particular, for any $n\in\N$ and $\sigma\in\N^n$, under the condition $\sigma\in\Sigma_n$, the RIFS $\F_{\sigma}$ has the same distribution as the primary RIFS $\F$.
        \item[(iii)]
Below we will often use special versions of the following relationship which results from (ii) via conditional expectations: For any random variable $X$ on the basic probability space with values in a measurable space $(E,\mathcal{E})$, $n\in\N$, $\sigma\in\N^{<\infty}$ and for any integrable function $h:(0,\infty)\times E\to\R$,
\begin{equation} \label{cond_exp_equation_1}
\E \left(\1\{\s\in\Sn\}h(\bfr_\sigma,X^{[\sigma]})\right)
=\P(\s\in\Sn)\,\int \E  h(\rho,X)\, \P_{\bfr_\sigma}(d\rho|\sigma\in\Sigma_n),
\end{equation}

where $\P_{\bfr_\sigma}(\cdot|\sigma\in\Sigma_n)$ is the conditional probability distribution of $\bfr_\sigma$ under the condition $\sigma\in\Sigma_n$. In particular, if $E=\R$ and $h(\rho,X)=g(\rho)X$ for some $g:(0,\infty)\to\R$, we obtain
\begin{equation} \label{cond_exp_equation}
\E \left(\1\{\s\in\Sn\}g(\bfr_\s)X^{[\sigma]}\right)
=\E \left(\1\{\s\in\Sn\} g(\bfr_\s)\right)\, \E X.
\end{equation}
\end{enumerate}
\end{rems}

Given some tree $\T$ as in Definition~\ref{def:codetree}, a random fractal set $F$ is obtained by applying at each construction step $n\in\N$ all the maps $\bff_\s$ of that step. More precisely, we set $\P$-a.s.
\begin{\eq}\label{F}
F:=\bigcap_{n=1}^\infty\bigcup_{\sigma\in\Sn}\bff_\sigma(\overline{O}).
\end{\eq}

The underlying labeled random tree $\T$ may interpreted as the \emph{generating tree} of the random fractal $F$.

In view of \eqref{F}, the random fractal set $F$ has a.s.\ the representation
\begin{\eq}\label{self-sim}
F=\bigcup_{i=1}^N f_i(F^{[i]})\,.
\end{\eq}
By assumption (A2), under the condition $i\in\Si_1$ the random set $F^{[i]}$ has the same distribution as $F$ and is independent of the corresponding random similarity mapping $f_i$. For different $i$'s the sets $F^{[i]}$ may depend on each other. Therefore, our model of random code tree fractals is an extension of \emph{stochastic self-similarity}.
Observe that also the following classes of fractals are special cases of our model:
\begin{enumerate}
  \item[1.] If one assumes ((A1) and) that the family $\{\F_{\s}:\s\in \St\}$ of random IFS is i.i.d., then $F$ is a \emph{random recursive self-similar set} in the sense of \cite{Fa86,Gr87,MW86}. Observe that in this case condition (A2) is satisfied. Hence the class of random recursive self-similar sets is included in our setting.
  \item[2.]   If one requires that for each $n\in\N_0$ the RIFS within the families $\{\F_{\s}: \s\in\Sn\}$ are all identical (i.e., there is only one RIFS at each level $n$), and if independence is assumed between the levels, then $F$ is a \emph{homogeneous random fractal} as considered in \cite{Ha92} or \cite{RWZ23}. Note that also in this case condition (A2) is satisfied and hence homogeneous fractals are included in our setting.
  \item[3.] Also \emph{$V$-variable fractals} are included in our model. Let $V\in\N$ and assume that at each level $n\in\N_0$, one has exactly $V$ i.i.d.\ copies $\G_{n,1},\ldots,\G_{n,V}$ of the RIFS $\F$ given, where independence is also assumed between the levels. Then for each vertex $\s\in\Sn$, one chooses one of them based on the \emph{type} $t(\s)$ of $\sigma$. The type $t(\s)$ of a vertex $\s\in\St$ is a random variable taking values in $\{1,\ldots,V\}$, assumed to be identically distributed for each $\s\in\St$. For $n\in\N_0$ and $\s\in\Sn$, one chooses $\F_{\s}:=\G_{n,t(\s)}$.

      If one assumes additionally that the family of types $\{t(\sigma):\sigma\in\St\}$ is independent and independent of the family $\{\G_{n,v}: n\in\N_0,v\in\{1,\ldots,V\}\}$ and if all types are assumed to be uniformly distributed, then this reproduces the $V$-variable model of Barnsley, Hutchinson and Stenflo \cite{BHS08,BHS12}.
      If we allow more general distributions for the types and more dependence, then we also recover the generalization of this model studied in \cite{Za23}. (Note that in \cite{Za23} the types are chosen together with the RIFS's $\G_{n,1},\ldots,\G_{n,V}$. It is assumed that for all $v\in\{1,\ldots,V\}$ and all $n\in\N_0$, the pair $(\G_{n,v},t_{n,v})$ has the same joint distribution, where $t_{n,v}$ is a vector of types of the same length as the vector $\G_{n,v}$, i.e., a type is chosen for each mapping in $\G_{n,v}$. No independence is assumed within the family $\{(\G_{n,v},t_{n,v}), v\in\{1,\ldots,V\}\}$ but these families are assumed to form an independent sequence w.r.t.\ $n$. Then for any $n\in\N_0$, $\s\in\Sn$ and $i\in\{1,\ldots,N_\s\}$ one sets $t(\s i):=t_{n,t(\s)}(i)$ (where $t_{n,v}(i)$ denotes the $i$th entry of the vector $t_{n,v}$) and $\F_{\s i}:=\G_{n+1,t(\s i)}$. This recovers the model considered in \cite{Za23}.

\end{enumerate}
Before turning to the main results (see Section~\ref{sec:main_res}), we illustrate the flexibility of our model with some examples.

\section{Examples}\label{Examples}

The aim of this section is to demonstrate that our setting allows to consider random fractals beyond the known classes described above.

If the independence between different levels in the tree is preserved, then the first equation in condition (A2), the independence of $f_i$ and $\T^{[i]}$ is satisfied automatically and we have a lot of freedom for depencencies within the levels. We can choose  for instance an arbitrary dependency structure for the random IFS at the first level, i.e. an arbitrary joint distribution. It is only important that each single RIFS has the same distribution as the primary RIFS $\F$, i.e. the marginal distributions need to coincide with that of $\F$. In order to satisfy the second equation in (A2), it is then necessary to repeat this dependency structure  in corresponding subfamilies at the subsequent levels. More precisely, for any $\s\in\Sigma_*$, the family $(\F_{\s i}: i\leq N_\s)$ should have the same joint distribution as the family $(\F_i: i\leq N_\emptyset)$. Between these families we could introduce at each level further dependencies, which we would then have to copy to subsequent levels to keep (A2) satisfied and so on.

We illustrate this with a very simple example, in which the primary IFS is chosen randomly from only two possible options, see Example~\ref{ex:dependent-SG}.

One can also construct tree fractals for which there are dependencies between different levels, see Examples~\ref{ex:3} and \ref{ex:percolation}.
While the first one is rather generic and only specifies some possible dependency structure, the second one is a very specific random carpet construction (a modification of fractal percolation). The intension was to use the dependencies to improve the connectivity within the fractal (compared to the independent case) without increasing the fractal dimension.

\begin{ex} \label{ex:dependent-SG}
Let $G=(g_1,g_2,g_3)$ be the IFS of the standard Sierpi\'nski gasket in $\R^2$, i.e.\ $g_i(x)=\frac 12 x+ t_i$, $x\in \R^2$, $i=1,2,3$, where $t_1=(0,0)$, $t_2=(\frac 12, 0)$ and $t_3=(\frac 14, \frac{\sqrt{3}}4)$, and let $G'=(g_1,\ldots,g_4)$ be the IFS of the corresponding equilateral triangle $T:=\conv\{(0,0),(1,0),(\frac 12, \frac{\sqrt{3}}2)\}$, i.e. $g_4(x)=\frac 12 A x+(\frac 12,0)$, where $A$ is the rotation by $\frac \pi 3$, see also Figure~\ref{fig:depSG1}.
The primary RIFS $\F$ is defined by choosing $G$ and $G'$ each with probability $\frac 12$.

To construct $F$, let $W:=\bigcup_{n\in\N_0}\{1,2,3,4\}^n$. We need two i.i.d.\ sequences $(M_n)_{n\in\N}$ and $(\G_\s)_{\s\in W}$,
independent of each other, where $M_n$ is a random number with uniform distribution on $\{1,2,3\}$ and $\G_\s$ has the same distribution as $\F$.

Given that the RIFS up to level $n-1$ have been chosen (such that the words at level $n$ are determined), the RIFS's at level $n$ are obtained as follows:

For each $\s\in\Sigma_{n-1}$,
set
$$
\F_{\s i}:=\begin{cases}
\G_\s & \text{ if } i=M_n,\\
\overline{\G}_\s & \text{ if } i\neq M_n \text{ and  } i\leq N_\sigma,
\end{cases}
$$
where  $\overline{\G}_\s$ is the IFS not chosen for $\G_\s$.

First note that the RIFS satisfies UOSC for the open set $O=\INt(T)$, since both IFS, $G$ and $G'$, satisfy OSC for $O$.
Secondly, it is easy to check that for any $\tau\in\Sigma_*$ the distribution of $\F_\tau$ is the same as that of $\F$. (Indeed, $\P(\F_\tau=G)=\frac 12=\P(\F_\tau=G')$.)

We claim that conditions (A1) and (A2) are satisfied.
For condition (A1) this is obvious, since the words of level $n$ are constructed recursively as daughters of words of level $n-1$. The independence of $f_i$ and $\T^{[i]}$ follows, since these random variables are generated by the RIFSs of different levels and the levels are independent by construction. Finally, for the equality in distribution of $\T^{[i]}$ and $\T$, observe that $\T=g((M_n)_{n\in\N},(\G_\s)_{\s\in W})$ for some measurable function $g$ and
$\T^{[i]}=g((M_n)_{n\geq 2},(\G_{i\s})_{\s\in W})$ for the same $g$. Since the sequences $(M_n)_{n\in\N}$ and $(\G_\s)_{\s\in W}$ e i.i.d.\ the subsequences $(M_n)_{n\geq 2}$ and $(\G_{i\s})_{\s\in W}$ are equal in distribution to the corresponding full sequences, and since they are independent this is enough to conclude
that $g((M_n)_{n\in\N},(\G_\s)_{\s\in W})=g((M_n)_{n\geq 2},(\G_{i\s})_{\s\in W})$ in distribution. Thus (A2) holds.

Let $F$ be the corresponding random fractal generated as in \eqref{F}. Observe that realizations of $F$ will show the following pattern: Out of the three (or four) daughters of a building block $T_\s=\bff_\s(T)$ at level $n-1$ exactly one (namely $T_{\s M_n}$, the copy at the corner indexed by $M_n$) will be different from all the others, while all the other daughters coincide. Moreover, the exceptional daughter will be in the same position for each mother of level $n-1$ (namely at the position indexed by $M_n$), see also Figure~\ref{fig:depSG2} for an illustration of the first construction steps.
This goes clearly beyond the homogeneous setting, while the independence of the different levels is preserved.
\end{ex}
\begin{figure}
  \includegraphics[width=0.45\textwidth]{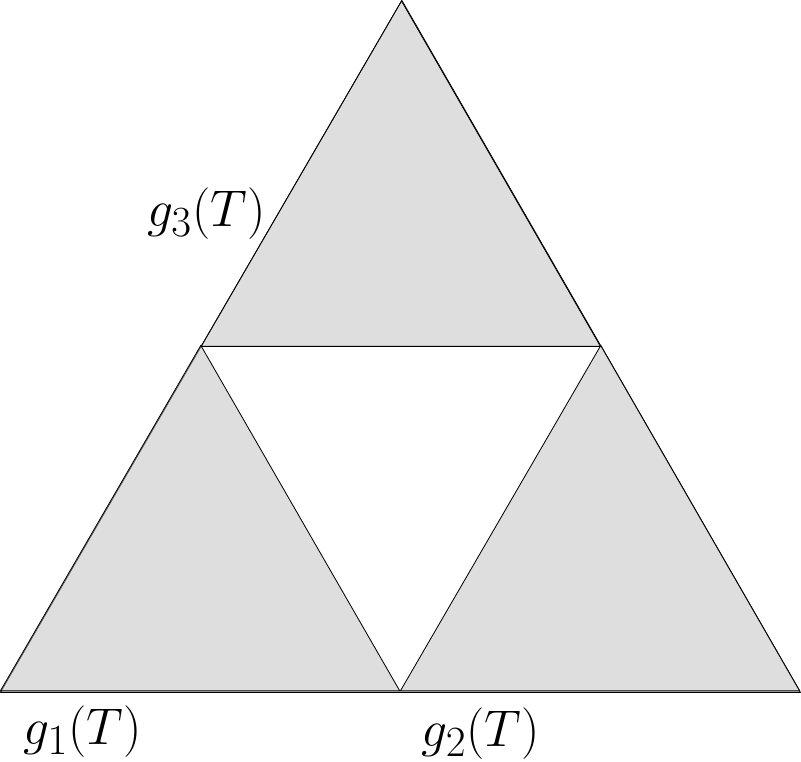}\hfill
  \includegraphics[width=0.45\textwidth]{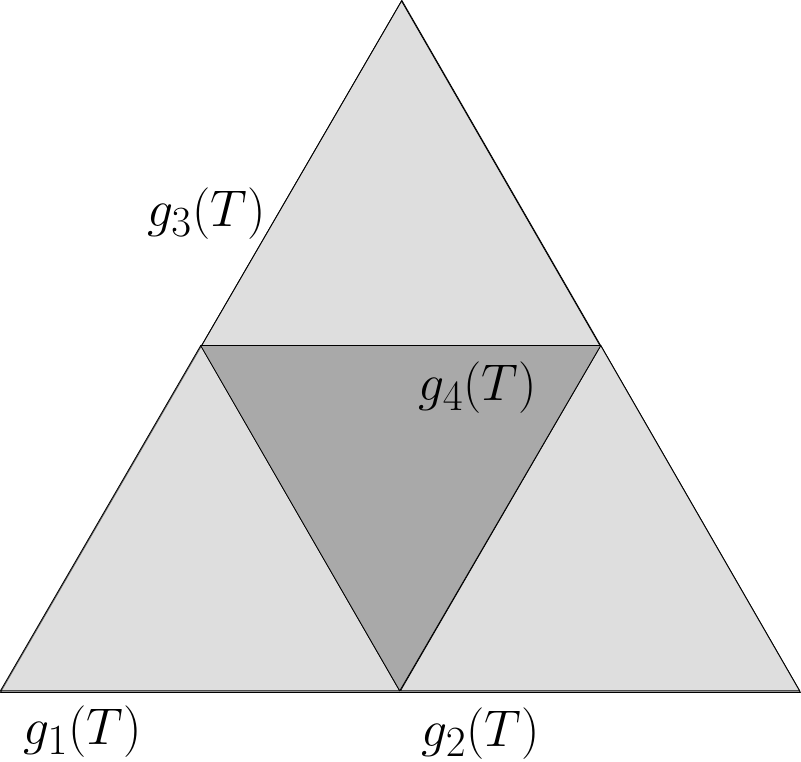}
  \caption{\label{fig:depSG1} Illustration of the first construction step of the two IFS $G$ and $G'$ in Example~\ref{ex:dependent-SG} used to generate a random Sierpi\'nski gasket with dependencies within the levels.}
\end{figure}

\begin{figure}
  \includegraphics[width=0.3\textwidth]{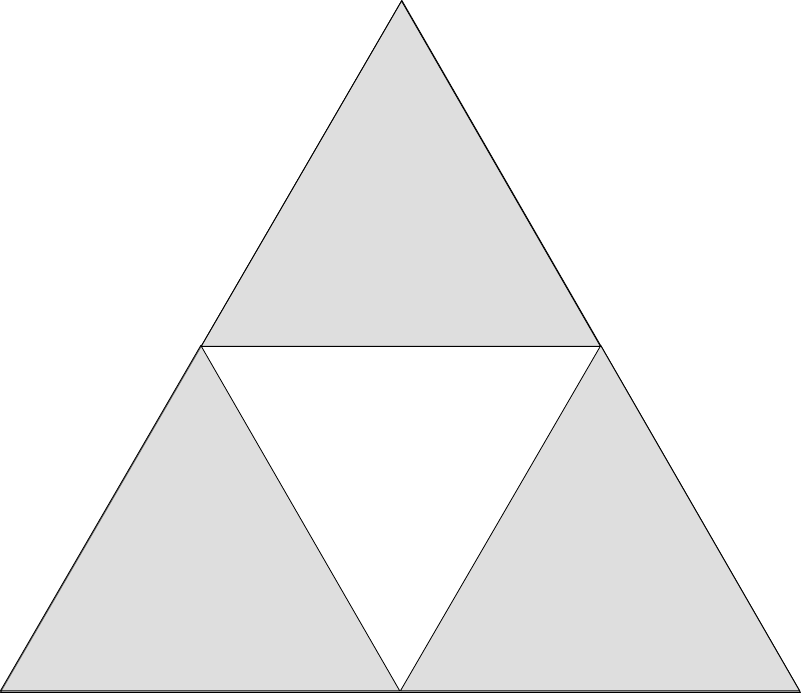}\hfill
  \includegraphics[width=0.3\textwidth]{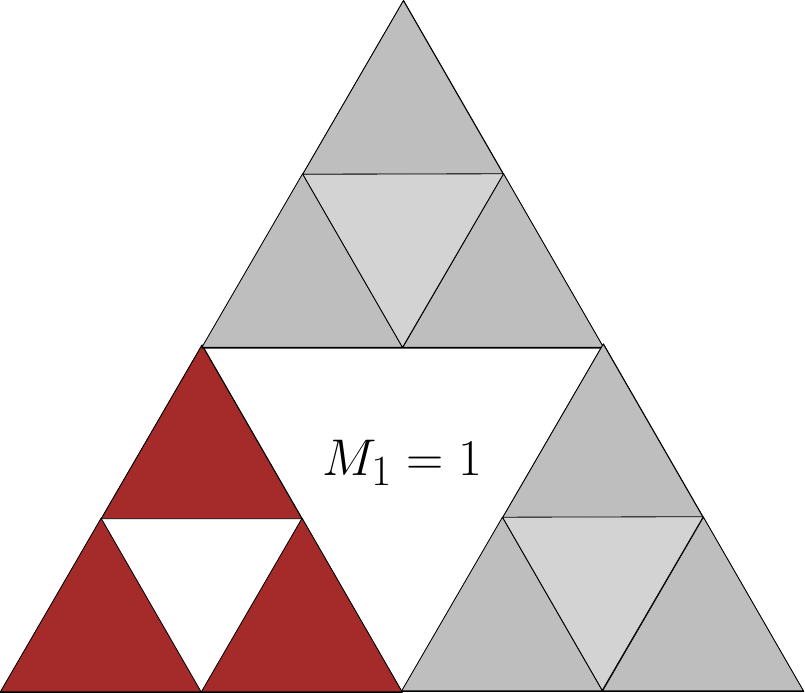}\hfill
  \includegraphics[width=0.3\textwidth]{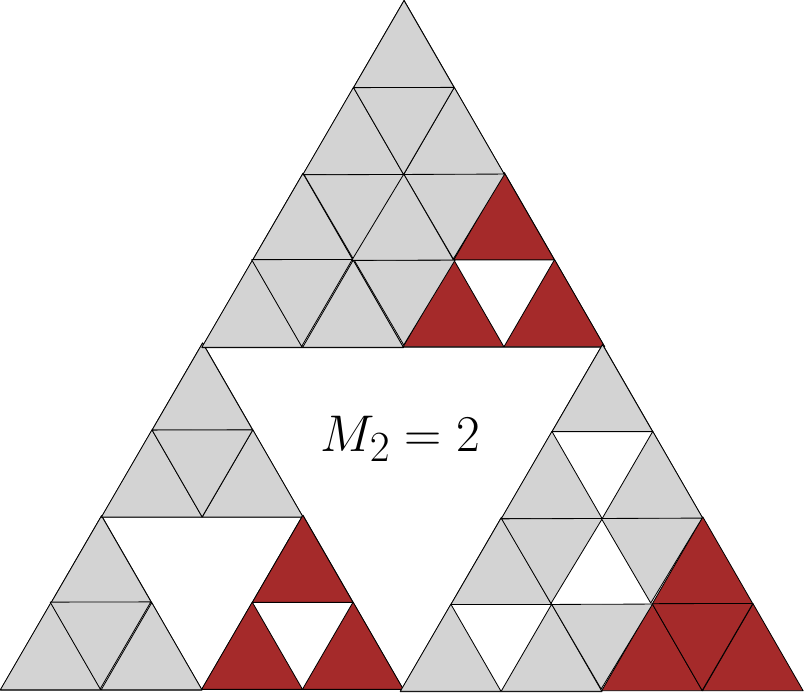}
  \caption{\label{fig:depSG2} Realization of the the first three construction steps of the random Sierpi\'nski gasket in Example~\ref{ex:dependent-SG}. The brown color indicates the 'exceptional' corner at each step (determined by the number $M_n$).}
\end{figure}

In the homogeneous or $V$-variable model, as well as in Example~\ref{ex:dependent-SG}, the RIFS's $\F_{\s},\F_{\tau}$ are independent if the lengths of $\sigma$ and $\tau$ are different (in other words, dependencies are allowed only within the same `time' level). Our model assumptions are, however, general enough to allow also dependencies across the time levels, as the next two examples show.
\begin{ex} \label{ex:3}
For simplicity, consider the case when the number $N$ of mappings in $\F$ is constant. Then $\Sigma_*=\bigcup_n \{1,\ldots,N\}^n$ is deterministic. Choose any finite set $W\subset\Sigma_*$ of nodes such that there is no pair of different elements of $W$ for which one element lies on the path to the root of the other element. We can consider a model where $\F_{\sigma\tau}=\F_{\sigma\tau'}$ whenever $\sigma\in\Sigma_*$ and $\tau,\tau'\in W$, and all other IFS's are independent. As a concrete example, consider $N=2$ and $W=\{1,21\}$. Our model can be constructed by induction as follows. Having $\F_{\sigma}$ constructed for all $\sigma$ with $|\sigma|\leq n$, set for any $\tau$ with $|\tau|=n-1$: $\F_{\tau 21}:=\F_{\tau  1}$, and choose $\F_{\tau 11}$, $\F_{\tau 12}$ and $\F_{\tau 22}$ independent of all IFS's chosen so far.
\end{ex}

\begin{ex} \label{ex:percolation}
\begin{comment}
Denote by $g_i:\R^2\to \R^2$, $i=1,\ldots,4$ the similarities
mapping the unit square $Q:=[0,1]^2$ (rotation and reflection free) to one of the four subsquares of sidelength $\frac 12$ as depicted in Figure~\ref{fig:Ex2} (left).
For $i=1,\ldots 4$, let $G_i:=\{g_j:j\neq i\}$ be the IFS consisting of the three mappings with indices different from $i$. That is, if the mappings of $G_i$ are applied to $Q$, then three subsquares are retained and the fourth, $g_i(Q)$, is discarded, see Figure~\ref{fig:Ex2} (right).
Our primary RIFS $\F$ is defined by choosing uniformly one of the $G_i$, $i=1,\ldots,4$. That is, each $G_i$ has  equal probability $\frac 14$. Note that UOSC \eqref{UOSC} is satisfied for the open set $O$ chosen to be the interior of $Q$.

\end{comment}

\begin{figure}
\includegraphics[width=0.22\textwidth]{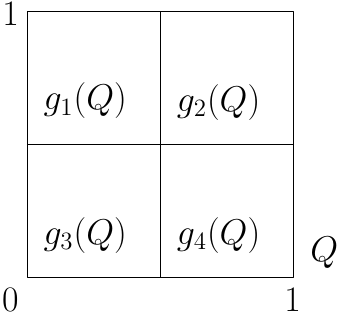}\hfill
\includegraphics[width=0.15\textwidth]{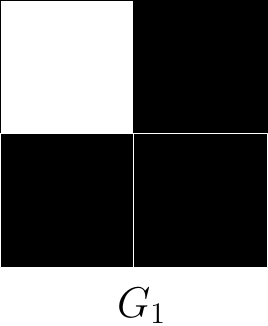}\hfill
\includegraphics[width=0.15\textwidth]{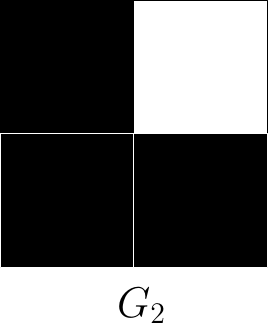}\hfill
\includegraphics[width=0.15\textwidth]{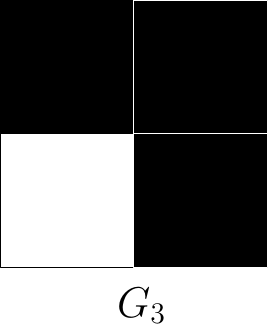}\hfill
\includegraphics[width=0.15\textwidth]{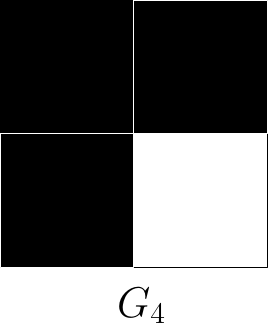}
  \caption{\label{fig:Ex2} The mappings $g_1,\ldots, g_4$ used in the IFSs $G_1,\ldots G_4$ in Example~\ref{ex:percolation}. Depicted are the images $g_i(Q)$ of the unit square $Q$ (left) and the unions of the images of $Q$ under the mappings of the IFSs $G_1,\ldots, G_4$ (right).}
\end{figure}

\begin{table}
\begin{tabular}{|c||c|c|c|c|}
  \hline
  & $\genfrac{}{}{0pt}{2}{\square\blacksquare}{\blacksquare\blacksquare}$
  & $\genfrac{}{}{0pt}{2}{\blacksquare\square}{\blacksquare\blacksquare}$
  & $\genfrac{}{}{0pt}{2}{\blacksquare\blacksquare}{\square\blacksquare}$
  & $\genfrac{}{}{0pt}{2}{\blacksquare\blacksquare}{\blacksquare\square}$ \\[1mm]
  & $p^l_1$  & $p^l_2$  & $p^l_3$ & $p^l_4$\\
    \hline\hline
  \begin{tabular}{c}
    $\genfrac{}{}{0pt}{2}{\square\blacksquare}{\blacksquare\blacksquare}\quad\genfrac{}{}{0pt}{2}{\blacksquare\blacksquare}{\square\blacksquare}\quad{ \square}$\\
    {$l=0,1,3$}\\
  \end{tabular} &
  1/4 & 1/4 & 1/4 & 1/4\\
  \hline
  \begin{tabular}{c}
    $\genfrac{}{}{0pt}{2}{\blacksquare\square}{\blacksquare\blacksquare}$\\
    $l=2$\\
  \end{tabular} &
  1/2 & 1/4 & 0 &1/4\\
  \hline
  \begin{tabular}{c}
    $\genfrac{}{}{0pt}{2}{\blacksquare\blacksquare}{\blacksquare\square}$ \\
    $l=4$\\
  \end{tabular} &
   0 & 1/4 & 1/2 & 1/4 \\
  \hline
\end{tabular}
\caption{The table shows the probabilities
$p^{l}_j$ for $l\in\{0,1,\ldots,4\}$ and $j\in\{1,\ldots,4\}$, which are used to choose an IFS in some square $Q$ based on the IFS already chosen in its left neighbor. \label{table1}}
\end{table}

Denote by $g_i:\R^2\to \R^2$, $i=1,\ldots,4$ the similarities
mapping the unit square $Q:=[0,1]^2$ (rotation and reflection free) to one of the four subsquares of sidelength $\frac 12$ as depicted in Figure~\ref{fig:Ex2} (left).
For $i=1,\ldots 4$, let $G_i:=\{g_j:j\neq i\}$ be the IFS consisting of the three mappings with indices different from $i$. That is, if the mappings of $G_i$ are applied to $Q$, then three subsquares are retained and the fourth, $g_i(Q)$, is discarded, see Figure~\ref{fig:Ex2} (right).
Our primary RIFS $\F$ is defined by choosing uniformly one of the $G_i$, $i=1,\ldots,4$. That is, each $G_i$ has  equal probability $\frac 14$. Note that UOSC \eqref{UOSC} is satisfied for the open set $O$ chosen to be the interior of $Q$.

We consider the (deterministic) tree
$$\Sigma_*:=\{1,2,3\}^{<\infty}$$
and we define the random labeled tree $\T$ by induction, using the notation
$$\T_n:=\{(\sigma,\F_\sigma):\, \sigma\in\Sigma_n\},\quad n=0,1,\dots $$
for the \emph{level $n$ of} $\T$.
We set $\F_\emptyset\ed\F$.
Given $n\in\N$, assume that the levels $\T_0,\dots,\T_{n-1}$ are given. Note that these levels of the tree determine the compact set
$$F_n:=\bigcup_{\sigma\in\Sigma_n}Q(\sigma),$$
where $Q(\sigma):=\bff_\sigma(Q)$, $\sigma\in\Sigma_*$, and vice versa: the random set $F_n$ determines $\T_0,\dots,\T_{n-1}$. So, given $F_n=K_n$ (for some deterministic configuration $K_n$), we define the `left-neighbor function' w.r.t.\ $K_n$
$$l_n:\Sigma_n\to\Sigma_n\cup\{0\}$$
so that if $\sigma\in\Sigma_n$ then $Q(l_n(\sigma))$ is the left adjacent square of $Q(\s)$ (in the lattice $2^{-n}{\mathbb Z}^2$), provided it lies in $K_n$. If the left adjacent square of $Q(\s)$ does not belong to $K_n$, then we set $l_n(\sigma):=0$.
We define the random IFS $\F_\sigma:\, \sigma\in\Sigma_n$ as a Markov chain (conditioned on $F_n=K_n$), with transition probabilities
\begin{align*}
&\P\left(\F_\sigma =G_k\mid \F_{l_n(\sigma)}=G_i,\,F_n=K_n\right)=:p^i_k, \\
&\P\left(\F_\sigma =G_k\mid l_n(\sigma)=0,\,F_n=K_n\right)=:p^0_k, \quad i,k\in\{1,2,3,4\},
\end{align*}
listed in Table~\ref{table1}.
This means that $\F_\sigma$ depends on $(\F_\tau:\tau<\sigma)$ only through its predecessor $l_n(\sigma)$ (here `$<$' is the lexicographic order of squares in the cubic lattice $2^{-n}{\mathbb Z}^2$). In fact, each horizontal line is a Markov chain, and the horizontal lines are mutually independent.
The distribution of $(\F_\sigma:\, \sigma\in\Sigma_n)$ under condition $F_n=K_n$ is then given by
\begin{equation}  \label{joint_distr}
    \P\left( (\F_\sigma=G_{\iota(\sigma)}:\, \sigma\in\Sigma_n)\mid F_n=K_n\right)=\prod_{\sigma\in\Sigma_n}p^{\iota(l_n(\sigma))}_{\iota(\sigma)},
\end{equation}
for any values $\iota(\sigma)\in\{1,2,3,4\}$, setting $\iota(0):=0$.

We shall show that our model is a random self-similar code tree satisfying the back path property. Since the number $N=3$ of branches is constant, the tree $\Sigma_*=\{1,2,3\}^{<\infty}$ is deterministic and (A1) is obviously fulfilled.
In order to verify (A2), we have to show that the tree $\T^{[i]}$ is independent of $f_i$ and has the same distribution as $\T$ for any $i\leq 3$.
Note that $f_i$ can take only two values, $g_i$ or $g_{i+1}$. In our situation, condition (A2) can  equivalently be expressed as the equality of distributions
$$
\P\left(\T^{[i]}\in (\cdot)\mid f_i=g_k\right)=\P\left(\T\in (\cdot)\right),\quad i=1,2,3,\quad k=i,i+1.
$$
Since the distribution of a labeled random tree is determined by the distribution of its root and the conditional distributions of its levels given the lower levels (see the above description for $\T$), (A2) will follow from the following two relations:
\begin{align}
    &\P(\F_i=G_j\mid f_i=g_k) = \P(\F=G_j), \label{A2-1}\\
    &\P\left( (\F_{i\sigma} = G_{\iota(\sigma)},\sigma\in\Sigma_n) \mid f_i=g_k,\, g_k(F_n)=g_k(K_n)\right) \label{A2-2}\\
    &\hspace{3cm}= \P\left( (\F_{\sigma} = G_{\iota(\sigma)},\sigma\in\Sigma_n) \mid F_n=K_n\right) ,\nonumber
\end{align}
for all $i=1,2,3$, $j,k=1,2,3,4$, $n\in\N$, $\iota:\Sigma_n\to\{1,2,3,4\}$ and possible realizations $K_n$ of $F_n$.
Property \eqref{A2-1} can be easily verified from the transition probabilities given in Table~\ref{table1}. As concerns \eqref{A2-2}, the right-hand side is given by \eqref{joint_distr}. We can apply \eqref{joint_distr} also for the left-hand side where, however, only selected IFS from the tree $\T$ at level $n+1$ are considered. Assume first that $k=1$ or $3$. Then, the subsquare $g_k(Q)$ of $Q$ lies in the left half of $Q$, and the key observation in this case is that the corresponding left neighbor functions are the same, i.e.,
$$l_{n+1}(i\sigma)=l_n(\sigma),\quad \sigma\in\Sigma_n$$
($l_n$ is defined w.r.t.\ $K_n$, whereas $l_{n+1}$ w.r.t.\ $g_k(K_n)$). Thus, we obtain the same expression using \eqref{joint_distr} also for the left-hand side of \eqref{A2-2}. In the case $k=2$ or $4$, we can apply the time reversal for the Markov chains in the lines (so that the lines will be ordered from the right to the left). One easily finds that this reversed Markov chain will have transition probabilities
\begin{align*}
p^k_i&=\P\left(\F_\sigma =G_k\mid \F_{r_n(\sigma)}=G_i,\, F_n=K_n\right), \\
p^0_k&=\P\left(\F_\sigma =G_k\mid r_n(\sigma)=0,\, F_n=K_n\right), \quad i,k\in\{1,2,3,4\},
\end{align*}
where $r_n$ denotes the right-neighbor function of $\Sigma_n$ w.r.t. $K_n$. Then, we can use the same argument as in the first case and verify so \eqref{A2-2}.

In contrast to the homogeneous or $V$-variable model, the different levels of this random labeled tree are \emph{not} independent. Indeed, consider the conditional probabilities
$$w_k:=\P\left(\F_1=G_1,\F_2=G_2,\F_3=G_3\mid \F_\emptyset=G_k\right).$$
If $k=1$ or $2$ then the configuration $G_2$ will lie left to $G_3$ in the bottom line, hence, $w_1=w_2=0$. If, however, $k=3$ or $4$, these two configurations will appear in different lines and we get $w_3=w_4=4^{-3}$, using Table~\ref{table1}.

\medskip

\end{ex}

\section{Markov stops and statement of the main results}
\label{sec:main_res}

We will use some familiar notions in the code space $\St$. By definition, for $\s\in\St$, the \emph{length} $|\s|$ is $n$, if $\s\in\Sn$, and
$\s|n$ is the \emph{restriction} of $\s$ to the first $n$ components, if $|\s|\geq n$. If $\s\in\St$ and $\tau\in\St^{[\s]}$, then $\s\tau\in\St$ is the concatenation of these codes.

Recall from \eqref{eq:bfs} the definition of the mappings $\bff_\s$ with contraction ratios $\bfr_\s$. By induction we infer from the stochastic self-similarity \eqref{self-sim} that a.s.
\begin{\eq}\label{sss-n}
F=\bigcup_{\s\in\Sn} \bff_\s(F^{[\s]})\,,\,n\in\N\,,\quad\text{ and}\quad
F=\bigcup_{\s\in\St} \bff_\s(F^{[\s]})\,,
\end{\eq}
where $\bff_\emptyset$ means the identity.
Recall from (A2) that under the condition $\s\in\Sn$ ($\s\in\St$, resp.) the random set $F^{[\s]}$ is independent of the random mapping $\bff_\s$ and has the same distribution as $F$. We will use the abbreviations
$$F_\sigma:=\bff_\sigma(F^{[\s]}),\quad \sigma\in\Sigma_*,$$
(including $F_\emptyset=\bff_\emptyset(F)=F$) so that \eqref{sss-n} becomes
$$
F=\bigcup_{\s\in\Sn} F_\sigma\,,\,n\in\N\,,\quad\text{ and}\quad
F=\bigcup_{\s\in\St} F_\sigma .
$$
For a boundedness condition in the application of the Renewal theorem below we will further use a formula similar to \eqref{sss-n} with respect to some Markov stop: Fix an arbitrary constant $R>\sqrt{2}\,|O|$ for the set $O$ from UOSC \eqref{UOSC} and define for all $0<r<R$ a random subset of codes by
\begin{\eq}\label{subtree-r}
\Si(r):=\left\{\s\in\St: R\,\bfr_\s\le r<R\,\bfr_{\s\mid|\s|-1}\right\},
\end{\eq}
where, by convention, $\bfr_{\s||\s|-1}=1$, if $|\s|=1$. It is convenient to set
$$
\Si(r):=\Si_0=\{\emptyset\} \quad\text{ for } r\geq R.
$$
Recalling that $\bff_\emptyset$ is the identity, we have, for any $r>0$,
\begin{\eq}\label{ssstop}
F=\bigcup_{\s\in\Si(r)} F_\s, \qquad \P\, \text{-a.s.}.
\end{\eq}
In order to formulate our main result in the most general form we also need the following random sets of \emph{boundary codes}, i.e., codes $\s\in\Si(r)$ for which the parallel set
$F_\s(r)$ has distance less than $r$ to the {\it boundary} of the first iterate
\begin{equation}  \label{fO}
\mathfrak{f}(O):=\bigcup_{i=1}^Nf_i(O)
\end{equation}
of the basic open set $O$ under the random similarities:
\begin{\eq}\label{btree}
\Sb(r):= \left\{\s\in\Si(r): F_\s(r)\cap(\mathfrak{f}(O)^c)(r)\ne\emptyset\right\},
\end{\eq}

Recall now that $1<\E N<\infty$. Let $D$ be the number determined by
\begin{\eq}\label{minkdim}
\E\sum_{i=1}^N r_i^D = 1\, .
\end{\eq}
Note that UOSC \eqref{UOSC} implies $D\le d$. The measure
\begin{\eq}\label{eq:mu-def}
\mu(\cdot):= \E \sum_{i=1}^N \1_{(\cdot)}(|\ln r_i|)\,r_i^D
\end{\eq}
is an associated probability distribution for the logarithmic contraction ratios $r_i$ of the primary random IFS $\F$. The corresponding mean value is denoted by
\begin{\eq}\label{def:eta}
\eta:=\E\sum_{i=1}^N |\ln r_i|\,r_i^D.
\end{\eq}
Under conditions (A1) and (A2), equation \eqref{minkdim} can be extended to any finite $n$ and to the Markov stops \eqref{subtree-r} as follows.
\begin{props}\label{dim_Markov stop}
Let $\T$ be a random self-similar code tree with back path independence (as in Definition~\ref{def:codetree}). Then, for any $n\in\N$ and any $r>0$,
$$
\E \sum_{\s\in\Sn}(\bfr_\s)^D=1 \quad \text{ as well as } \quad
\E \sum_{\s\in\Si(r)}(\bfr_\s)^D=1.$$
\end{props}
\begin{proof}
For $n=1$ the first equation is given by \eqref{minkdim}. The general case follows by induction: Assume that it is true for some $n\in\N$. Then we get
\begin{align*}
\E\sum_{\s'\in\Si_{n+1}}(\bfr_{\s'})^D
&=\E\sum_{\s\in\Sn}(\bfr_\s)^D\sum_{i=1}^{N_{\s}}(r_{\s i})^D
=\E\sum_{\s\in\Sn}(\bfr_\s)^D\,\E\sum_{i=1}^N r_i^D\\
&=\E\sum_{\s\in\Sn}(\bfr_\s)^D=1,
\end{align*}
by the induction assumption. In the second equality we have used \eqref{cond_exp_equation} with $X:=\sum_{i=1}^Nr_i^D$ and $g(\rho):=\rho^D$. For the third equality, the case $n=1$, i.e.~\eqref{minkdim}, is used. This proves the first assertion.

The arguments for the second one are similar. For fixed $r>0$ choose $m\in\N$ large enough so that $(r_{\max})^m\le r$. By means of the corresponding conditional expectations  and the subtree condition \eqref{cond_exp_equation} for $X:=\sum_{\sigma''\in\Sigma_{m-n}}(\bfr_{\sigma''})^D$, we infer from the first statement
\begin{align*}
1&=\E\sum_{\s\in\Si_m}(\bfr_\s)^D\\
&=\E\bigg(\sum_{n=1}^{m-1}\sum_{\substack{\s'\in\Si(r)\\ |\s'|=n}} (\bfr_{\s'})^D \sum_{\s''\in\Si_{m-n}^{[\sigma']}}\left( \bfr_{\s''}^{[\sigma']}\right)^D\bigg)
+\E\sum_{\substack{\s'\in\Si(r)\\|\s'|=m}}(\bfr_{\s'})^D\\
&=\E\bigg(\sum_{n=1}^{m-1}\sum_{\substack{\s'\in\Si(r)\\|\s'|=n}}(\bfr_{\s'})^D~
\E\big(\sum_{\s''\in\Si_{m-n}}(\bfr_{\s''})^D\big)\bigg)
+\E\sum_{\substack{\s\in\Si(r)\\|\s|=m}}(\bfr_{\s})^D\\
&=\E\bigg(\sum_{n=1}^{m-1}\sum_{\substack{\s'\in\Si(r)\\|\s'|=n}}(\bfr_{\s'})^D\bigg)+\E\sum_{\substack{\s\in\Si(r)\\|\s|=m}}(\bfr_{\s})^D=\E\sum_{\s\in\Si(r)}(\bfr_{\s})^D\,,
\end{align*}
which proves the second assertion.
\end{proof}
For our main result we need a slightly stronger condition than UOSC \eqref{UOSC}, namely the
{\it Uniform Strong Open Set Condition} (USOSC). It is satisfied, if
\begin{\eq}\label{USOSC}
\mbox{UOSC holds for some $O$ such that}~~ \P(F\cap O\ne\emptyset)>0\, .
\end{\eq}
(This inequality is equivalent to the commonly used condition
$\P(F\cap O\neq \emptyset)=1$, see \cite[Remark~3.1]{RWZ23}.)

By Lemma~\ref{lems:large distances}, the boundary $\partial F(r)$ of the parallel set $F(r)$ is $\P$-a.s.\ a $(d-1)$-Lipschitz manifold and $\widetilde{F(r)}$ has positive reach for all $r\geq R$, where $R>\sqrt{2}\,|O|\geq\sqrt{2}\,|F|$ is the constant from \eqref{subtree-r}. For $r<R$ one needs to impose some regularity assumption in general.
Recall that $\K$ denotes the family of all nonempty compact subsets of $\R^d$ and $\Reg$ is the set of all pairs $(r,K)$ such that $r>0$, $K\in\K$ and $r\not\in\crit K$ (i.e., $r$ is a regular value of the distance function $d_K$).
We shall call a random compact set $F$ \emph{regular} if
the set of all critical values of its distance function has Lebesgue measure zero almost surely, i.e., if
\begin{equation}\label{regular SSRS}
\cL(\crit F)=0\quad \text{a.s.}
\end{equation}
If $F$ is a regular random fractal set defined by means of \eqref{F}, we also consider the random set $\reg_*F\subset (0,\infty)$ defined by
\begin{\eq}\label{Reg*}
\reg_*F:=\left\{r>0: (r,F_\s)\in \Reg \text{ for all } \sigma\in\Sigma_{*}\right\}\, .
\end{\eq}
Recall that $\St=\bigcup_{n=0}^\infty\Si_n$
and note that
$(0,\infty)\setminus\Reg_*=\bigcup_{\sigma\in\Sigma_*}\crit F_\sigma$ and
$\crit F_\sigma=\bfr_\s\crit F^{[\sigma]}$, $\sigma\in\Sigma_*$. Hence,
\begin{align*}
    \E\cL\left((0,\infty)\setminus\reg_*F\right)
    &\leq\E \sum_{\sigma\in\Sigma_*}\cL(\crit F_\sigma)\\
    &\leq\sum_{n=0}^\infty\sum_{\sigma\in\N^n}\E\left(\1\{\s\in\Sn\}\bfr_\s\cL(\crit F^{[\sigma]})\right)\\
    &=\sum_{n=1}^\infty\sum_{\sigma\in\N^n}\E\left(\1\{\s\in\Sn\}\bfr_\s\right)\, \E\cL(\crit F)
    =0
\end{align*}
by \eqref{regular SSRS} (we have used \eqref{cond_exp_equation} with $X=\cL(\crit F)$). Thus we have
\begin{equation}\label{Reg*-0}
\cL\left((0,\infty)\setminus\reg_*F\right)=0\quad\text{ a.s.}
\end{equation}

Now we can formulate our main result. In the sequel, the occurring essential limits and suprema are always meant with respect to
Lebesgue-a.a.\ arguments. (Recall the notations $F_\s=\bff_\sigma(F^{[\sigma]})$, $\bff_\sigma=f_{\sigma_1}\circ f_{\sigma_1\sigma_2}\circ\dots\circ f_{\sigma_1\ldots\sigma_n}$, for $\s=\s_1\dots\s_n\in\Sn$, and, in particular, $\bff_i=f_i$ for $n=1$. Recall $\mu$ and $\eta$ from \eqref{eq:mu-def} and \eqref{def:eta}. The  special code sets $\Sigma(r)$ and $\Sigma_b(r)$, $r>0$, are defined in \eqref{subtree-r} and \eqref{btree}, resp.) We will also use the short notation $O_\sigma:=\bff_\sigma(O)$, $\sigma\in\Sigma_*$.

\begin{thms}\label{maintheorem}
Let $k\in\{0,1,\ldots,d\}$ and let $F$ be a random self-similar code tree fractal  in $\R^d$ (defined in \eqref{F}) satisfying the back path property (cf.\ Def.~\ref{def:codetree}) and the Uniform Strong Open Set Condition \eqref{USOSC} with basic set $O\subset\R^d$. Let $R>\sqrt{2}|O|$. For $k\le d-2$ we additionally suppose the following:
\begin{itemize}
\item[{\rm (i)}] if $d\geq 4$, then $F$ is regular in the sense of \eqref{regular SSRS},
\item[\rm (ii)] there exists a constant $c_k>0$ such that with probability one,
$$C_k^{\var}\left(F(r),O_\sigma(r)\cap O_\tau(r)\right)\leq c_kr^k$$
for a.a. $r>0$ and all $\sigma,\tau\in\Sigma(r)$ with $\sigma\neq\tau$.
\end{itemize}

Set for almost all $r>0$,
$$R_k(r):=\E C_k(F(r))-\E\sum_{i=1}^N\1_{(0,R r_i]}(r)\, C_k\big(F_i(r)\big)\,.$$
Then we get the following:
\begin{itemize}
\item[{\rm (I)}] If the measure $\mu$ is non-lattice, then
$$C_{k,F}^{\fr}:=\elim_{\ep\rightarrow 0}\limits\ep^{D-k}\E\,C_k(F(\ep))=\frac{1}{\eta}\int_0^{R} r^{D-k-1}R_k(r)\, dr\, .$$

\item[{\rm (II)}] If the measure $\mu$ is lattice with constant $c$, then for almost all \ $s\in[0,c)$
  $$\lim_{n\rightarrow\infty}e^{(k-D)(s+nc)}\E\,C_k\big(F(e^{-(s+nc)})\big)=\frac{1}{\eta}\sum_{m=0}^\infty e^{(k-D)(s+mc)}R_k\big(e^{-(s+mc)}\big).
  $$
	
\item[\rm{(III)}] In general,
$$\overline{C}_{k,F}^{\fr}:=\lim_{\delta\rightarrow 0}\frac{1}{|\ln\delta|}\int_\delta^1\ep^{D-k}\E\, C_k(F(\ep))~\ep^{-1}d\ep = \frac{1}{\eta}\int_0^{R} r^{D-k-1}R_k(r)\, dr.$$
\end{itemize}
\end{thms}

\begin{rems}\label{rem:partial}
As in \cite{RWZ23}, assumption (ii) can be replaced by the following less restrictive pair of assumptions:
\begin{itemize}
\item[{\rm (ii')}]
for all $r_0\in(0,R)$,
$$\E\esup_{r_0\le r\le R}\limits\,\max_{\s\in\Sb(r)}C_k^{\var}\bigg(F(r),\partial\big(F_\sigma(r)\big)\cap\partial\big(\bigcup_{\substack{\tau\in\Si(r),\\ \tau\ne\s}} F_\tau(r)\big)\bigg)<\infty\, ,$$
\item[{\rm (iii')}]
there is a constant $C>0$ such that with probability 1,
$$\E\bigg[\max_{\s\in\Sb(r)}\, r^{-k}\,C_k^{\var}\bigg(F(r),\partial \big(F_\sigma(r)\big)\cap\partial\big(\bigcup_{\substack{\tau\in\Si(r),\\ \tau\ne\s}} F_\tau(r)\big)\bigg)\bigg|\sharp(\Sb(r))\bigg]\le C$$
for Lebesgue almost all $r\in(0,R]$.
\end{itemize}
Moreover, in conditions (ii') and (iii'), the boundary signs $\partial$ can be omitted, since $\Int F_\sigma(r)\subset\Int F(r)$ for any $\sigma\in\St$ and the curvature measures are concentrated on the boundary of the set $F(r)$. The proof with these more general assumptions would follow the same lines as that of \cite[Theorem~3.2]{RWZ23}. Unfortunately, we do not know any example of a model where (ii) does not hold, but (ii') and (iii') hold.
\end{rems}

\begin{thms}\label{positive}
Under the conditions of Theorem \ref{maintheorem} we have
$$\liminf_{\varepsilon\to 0} \varepsilon^{d-D} \mathbb{E}\mathcal{L}^d(F(\varepsilon))>\,0.$$
This implies in particular $\overline{C}_{d,F}^{\fr}>0$ and, in case it exists, $C_{d,F}^{\fr}>0$. Moreover, if $D<d$, then
$$\overline{C}_{d,F}^{\fr}=\frac{2}{d-D}\overline{C}_{d-1,F}^{\fr}.$$
\end{thms}

\begin{rems} \label{rem:dimensions}

As a consequence of Theorem~\ref{positive}, the number $D$ (defined in \eqref{minkdim}) may be interpreted as the Minkowski dimension of $F$ \emph{in the mean sense}. This notion arises, when in the definition of the (upper and lower) Minkowski dimension, the (upper or lower) Minkowski content is replaced by its mean version. In general, we do not yet know whether the Minkowski dimension of $F$ (as a random variable) exists in our model. But in some cases it exists and is almost surely constant. Then this almost sure Minkowski dimension of $F$ can be strictly smaller than $D$.
For example, in the case of homogeneous random fractals with (UOSC) it follows from  \cite[Main Theorem]{BHS12} together with \cite[Corollary 2.26]{Tr17} that the former is determined by the unique $s$ such that
$\mathbb{E}\,\ln(\sum_{i=1}^N r_i^s)=0$, which is in general, less than the above $D$ given by $\mathbb{E}\sum_{i=1}^N r_i^D =1$. Moreover, the corresponding almost sure average Minkowski content does not exist, see  \cite{Tr23}.

\end{rems}

\section{Proofs}\label{Proofs}

In order to prepare the proof of Theorem \ref{maintheorem},
we define the random function $\xi_k:(0,\infty)\to\R$ by $$\xi_k(r):=\1_{(0,R]}(r)C_k(F(r))-\sum_{i=1}^N \1_{(0,R\,r_i]}(r)C_k(F_i(r)), \quad r\in \reg_*F,$$
where the constant $R$ is as in the theorem and $\reg_*F$ as in \eqref{Reg*}.  Note that $\xi_k$ is $\P$-a.s.\ determined for almost all $r>0$, see equation \eqref{Reg*-0}.  By the motion invariance and scaling property of $C_k$, we get for a.a. $r>0$,
\begin{align} \label{eq:xi}
\xi_k(r)=\1_{(0,R]}(r)C_k(F(r))- \sum_{i=1}^N \1_{(0,R]}(r/r_i)r_i^k C_k(F^{[i]}(r/r_i)).
\end{align}

Below we will prove that the expectations of the absolute values of all the summands on the right hand side are finite. Then, in view of \eqref{Reg*-0}, the function $R_k=\E \xi_k\, $
in Theorem \ref{maintheorem} is determined at Lebesgue-a.a.\ arguments. For this and further arguments, the following estimates for $\xi_k$ are required. A proof is provided later on page \pageref{proof-lem:main-xi_k}.
\begin{lem}
  \label{lem:main-xi_k}
 There exist constants $\delta>0$ and $a_k>0$ such that
\begin{equation} \label{bound_2}
\E |\xi_k(r)|\leq a_k r^{k-D+\delta}\quad\text{ for a.a. } r>0.
\end{equation}

Moreover, for all $0<r_0\le R$,
\begin{\eq}\label{locbound}
\E\esup_{r>r_0}\limits|\xi_k(r)|<\infty,
\end{\eq}
where $\esup$ means the supremum over all arguments for which the function is determined.

\end{lem}

\begin{proof}[Proof of Theorem \ref{maintheorem}] In order to translate the problem into the language of the renewal theorem, we substitute $r=Re^{-t}$ and define
\begin{align*}
  Z_{k}(t)&:=\1_{[0,\infty)}(t)e^{(k-D)t} C_k(F(Re^{-t})),
\end{align*}
whenever $(Re^{-t}, F)\in\Reg$.
Moreover, we define
\begin{align*}
  z_{k}(t)&:=e^{(k-D)t}\xi_k(Re^{-t}),
\end{align*}
whenever $Re^{-t}\in\reg_*F$.
Note that $z_k(t)=0$ for $t< 0$. We infer from the relation \eqref{eq:xi} that, for any $t\geq 0$ such that $Re^{-t}\in\reg_*F$,
$$Z_k(t)=\sum_{i=1}^N r_i^D Z_{k}^{[i]}(t-|\ln r_i|)+z_k(t).$$
Note in particular, that the right hand side is well defined for all such $t$. (For all $t<0$ the relation also holds trivially, since both sides are zero.)

Denote
\begin{equation} \label{def_T}
T:=\left\{t>0: \P(Re^{-t}\in\reg_*F)=1\right\}.
\end{equation}
Let $\mu^{\ast n}$ be the $n$th convolution power of the distribution $\mu=\E\sum_{i=1}^N \1_{(\cdot)}(|\ln r_i|)r_i^D$, if $n\geq 1$, $\mu^{\ast 0}$ the Dirac measure at $0$, and
\begin{\eq}\label{U}
U(t):=\sum_{n=0}^\infty\mu^{\ast n}((0,t]),~t>0\,.
\end{\eq}
Note that the summands on the right vanish for $n>t/|\ln r_{\max}|$, so that the summation is finite for each $t>0$. (In the sequel $U$-a.a.\ means a.a.\ with respect to the corresponding measure.)

We shall need the following properties of the set $T$ defined in \eqref{def_T}. A proof will be provided later.

\begin{lem}   \label{L_1}
Under the assumptions of Theorem~\ref{maintheorem} we have
\begin{equation} \label{T_full}
\cL((0,\infty)\setminus T)=0
\end{equation}
and
\begin{\eq} \label{ts}
t\in T \text{ implies } t-s\in T  \text{ for } \mu \text{-a.a. } s\leq t \text{   and for } U \text{-a.a. }s\leq t.
\end{\eq}
\end{lem}

Using \eqref{ts}, we infer for $t\in T$ from the above equality for $Z_k(t)$ that
\begin{align*}
\E|Z_k(t)|&\leq\E\sum_{i=1}^N r_i^D\left|Z_k^{[i]}\left(t-|\ln r_i|\right)\right|+\E|z_k(t)|\\
&=\sum_{i=1}^\infty\E\,\1(i\in\Sigma_1) r_i^D\left|Z_k^{[i]}\left(t-|\ln r_i|\right)\right|+\E|z_k(t)|\\
&=\int \E|Z_k(t-s)|\mu(ds)+\E|z_k(t)|\,.
\end{align*}
For the last equality, we have first used \eqref{cond_exp_equation_1} with $n=1$, $X:=Z_k$ (a random function of real variable) and $h(\rho,X):=|X(t-|\ln \rho|)|$, and then the definition of the measure $\mu$.
In view of \eqref{ts},
 we obtain from iterated application of this inequality that
 \begin{align}
   \label{eq:Z-U-bound}
   \E|Z_k(t)|\leq\int_0^t\E|z_k(t-s)|dU(s), \quad  t\in T.
 \end{align}

We conclude from \eqref{bound_2} in Lemma~\ref{lem:main-xi_k} that there exist constants $a'_k>0$ and $\delta>0$ such that
\begin{\eq}\label{bound}
\E|z_k(u)|\le a'_k\1_{[0,\infty)}(u)e^{-\delta u},\quad u\in T.
\end{\eq}

Since in our case $U(t)\le t/|\ln r_{\max}|$ for all $t>0$, we infer from \eqref{bound} and \eqref{eq:Z-U-bound} that for some constants $d_k$ and $d'_k$,
$$\E|Z_k(t)|\le \int_0^t \E|z_k(t-s)| dU(s)<d_k\,,~t\in T,$$
and consequently,
\begin{\eq}\label{expectation bound}
\E|C_k(F(r))|\leq d'_k\, r^{k-D}|\ln(r/R)|~,~~\mbox{for a.a.}~ 0<r\leq R.
\end{\eq}
This shows, in particular, the finiteness of the expectations mentioned at the beginning of the section.

Moreover, we can repeat to above arguments omitting the absolute value signs and replacing the corresponding inequalities by  equalities in order to obtain the renewal equation in the sense of Feller \cite{Fe71}. We get for all $t\in T$,
$$\E Z_k(t)=\int_0^t \E Z_k(t-s)\mu(ds)+\E z_k(t)$$
and
\begin{align}
   \label{eq:Z_k-repres} \E Z_k(t)=\int_0^t \E z_k(t-s)\, dU(s).
\end{align}

We now define two auxiliary functions on $(0,\infty)$ by
\begin{align*}
\overline{z}_k(t)&:=\begin{cases}
   \E z_k(t), &{\rm if}~t\in T,\\
   \limsup\limits_{\substack{t'\to t\\t'\in T}}\E z_k(t'), &{\rm if}~t\in (0,\infty)\setminus T,
\end{cases}\\
\overline{Z}_k(t)&:=\int_0^t\overline{z}_k(t-s)dU(s),~t>0.
\end{align*}
Then in view of \eqref{ts} and \eqref{eq:Z_k-repres},
\begin{align}\label{Zbar}
\overline{Z}_k(t)=\E Z_k(t) \quad \text{ for } t\in T.
\end{align}
Lemma~\ref{L_cont} implies now that the random function $\xi_k$ is continuous at all $r\in\reg_*F$.  (For $k=d,d-1$ we do not need (i) for this conclusion.)
The dominated convergence theorem, justified by \eqref{locbound}, yields that at each $t\in T$ the function $\overline{z}_k$, which agrees with $e^{(k-D)t}\E \xi_k(Re^{-t})$ at such $t$, is continuous. Since $T$ has full measure in $(0,\infty)$ (see \eqref{T_full}), $\overline{z}_k$ is continuous Lebesgue-a.e. on $(0,\infty)$. Moreover, in view of \eqref{bound}, $\overline{z}_k$ is bounded by a directly Riemann integrable function. Thus, according to  Asmussen \cite[Prop. 4.1, p. 118]{As87}, $\overline{z}_k$ is directly Riemann integrable, too. Therefore the classical renewal theorem in Feller \cite[p.\ 363]{Fe71} can be applied, which yields that
$$\lim_{t\rightarrow\infty}\overline{Z}_k(t)=\frac{1}{\eta}\int_0^\infty\overline{z}_k(t)\,dt.$$
The right hand side agrees with
$$\frac{1}{\eta}\int_0^\infty\E z_k(t)\,dt=\frac{1}{\eta}\int_0^\infty e^{(k-D)t}R_k(Re^{-t})\, dt,$$
since $\overline{z}_k(t)=\E z_k(t)$ for Lebesgue-a.a.\ $t$.
Multiplying $\j^{D-k}$ in this equation and substituting $r=R e^{-t}$ under the integral, assertion (I) follows in view of
\eqref{Zbar}.

Since $\E Z_k$ is bounded on finite intervals, in the non-lattice case the corresponding average limit in (III) is a consequence.

In the lattice case, the renewal theorem provides the limit along arithmetic progressions with respect to the lattice constant, here only for those sequences along which the function is determined. This shows assertion (II). The latter also implies the average convergence (III). For more details we refer to the arguments of Gatzouras at the end of the proof of \cite[Theorem~2.3]{Ga00} in the classical case.
This completes the proof of Theorem~\ref{maintheorem} up to proofs of Lemmas~\ref{lem:main-xi_k} and \ref{L_1}.
\end{proof}

\begin{proof}[Proof of Lemma~\ref{L_1}]
Using \eqref{Reg*-0}, Fubini's theorem and the definition of $T$, we get
\begin{align*}
    0&=\E\cL((0,R)\setminus\reg_*F) =\int_0^R \P(r\not\in\reg_*F)\, dr\\
    &=\int_0^\infty\P(Re^{-t}\not\in\reg_*F)\, Re^{-t}\, dt = \int_{(0,\infty)\setminus T} \P(Re^{-t}\not\in\reg_*F)\, Re^{-t}\, dt.
\end{align*}
Since the integrand of the last integral is positive, \eqref{T_full} follows.

To show \eqref{ts}, note that $t\in T$ implies $\P(Re^{-t}\in\reg_*F)=1$ and thus, for any $i\in\N$,
$$\P\left[\left. Re^{-(t-|\ln r_i|)}\in\reg_*F^{[i]}\right| i\leq N\right]=1,$$
which follows from the definition \eqref{Reg*} of $\reg_*F$ and the similarity property of the mappings $f_i$. Employing the random function $X(t):=\1\{Re^{-t}\in\reg_*F\}$, $t>0$, we infer that
\begin{align*}
1&=\E\sum_{i=1}^N r_i^D
=\sum_{i=1}^\infty\E\left(\1\{i\leq N\}r_i^D X^{[i]}(t-|\ln r_i|)\right)\\
&=\E\sum_{i=1}^N r_i^D X(t-|\ln r_i|)\\
&=\int X(t-s)\,\mu(ds).
\end{align*}
Here we have used \eqref{minkdim}, \eqref{cond_exp_equation_1} with $h(\rho,X):=\rho^DX(t-|\ln\rho|)$, and the definition of the distribution $\mu$.
Hence, $\P(Re^{-(t-s)}\in\reg_*F)=1$ for $\mu$-a.a. $s$ and therefore $t-s\in T$ for $\mu$-a.a. $s\leq t$, which proves the first assertion in \eqref{ts}.

From iterated application of this result and the  definition of $U$ (see \eqref{U}) we obtain that $t\in T$ implies $t-s\in T$ for $U$-a.a. $s\leq t$.
\end{proof}

In order to prepare the proof Lemma~\ref{lem:main-xi_k}, we state an estimate concerning the number of elements of the boundary code sets $\Sb(r)$ defined in \eqref{btree}.

The proof is analogous to the one of \cite[(4.15)]{RWZ23} for the homogeneous case, but it has to be adapted to our model assumptions. We also redefine slightly the set of words $\Xi(r)$, fixing this way a small gap in the proof in \cite{RWZ23}.

\begin{lem}  \label{L_2}
There is some constant $\delta$ with $0<\delta<D$ such that
\begin{align}\label{aux_est}
\sup_{0<r<R}\E r^{D-\delta}\sharp\left(\Sb(r)\right)<\infty.
\end{align}
\end{lem}

\begin{proof}
By USOSC, i.e., UOSC with $O$ such that $\P(F\cap O\ne\emptyset)>0$, there exist some constants $\alpha>0$ and $0<\r<1$ such that $\P(\Si(\r,\alpha)\ne\emptyset)>0$, where
\begin{\eq}\label{rhoalpha_}
\Si(\r,\alpha):=\left\{\tau\in\Si(\r): O_\tau(\alpha)\subset O\right\}.
\end{\eq}
By Proposition \ref{dim_Markov stop}, we have $\E\sum_{\s\in\Si(\r)}(\bfr_{\s})^D=1$.
Thus, there exists $0<\delta<D$ such that
\begin{\eq}\label{delta_}
\E\sum_{\tau\in\Si(\r)\setminus\Si(\r,\alpha)} (\bfr_\tau)^{D-\delta}=1\, .
\end{\eq}
Let $\Xi_*\subset\Sigma_*$ be the set of all finite words $\sigma\in\Sigma_*$ such that $\tau\not\in\Sigma^{[\sigma']}(\rho,\alpha)$ whenever $\sigma=\sigma' \tau\sigma''$, and denote
$$\Xi(r):=\Xi_*\cap\Sigma(r),\quad r>0.$$
Below we will show that there is some constant $Q>0$ such that, for all $r>0$,
\begin{equation}  \label{psi}
r^{D-\delta}\E\,\sharp\left(\Xi(r)\right)<Q.
\end{equation}

Note that if $\sigma\in\Sigma(r)$ is of the form $\sigma=i\sigma'\tau\sigma''$ with $i\leq N$, $\sigma'\in\Sigma^{[i]}_*$, $\tau\in\Sigma^{[i\sigma']}(\rho,\alpha)$ and $\sigma''\in\Sigma^{[i\sigma'\tau]}_*$ then
$$d(O_\sigma,\partial \mathfrak{f}(O))>r_i\bfr_{\sigma'}\alpha.$$
Let $m$ be the smallest integer with $r_{\max}^{m-1}<\alpha/(2\rho)$. We claim that, if $|\sigma''|\geq m$, then $\sigma\not\in\Sigma_b(r)$. Indeed, we have
$\bfr_{\sigma''|\,|\sigma''|-1}<\alpha/(2\rho)$, and using that $\tau\in\Sigma^{[i\sigma']}(\rho)$ and $\sigma\in\Sigma(r)$, we conclude that
\begin{align*}
    d(O_\sigma,\partial \mathfrak{f}(O))&>r_i\bfr_{\sigma'}\alpha
    >2Rr_i\bfr_{\sigma'}(\rho/R)\bfr_{\sigma''|\,|\sigma''|-1}\\
    &\geq 2Rr_i\bfr_{\sigma'}\bfr_\tau \bfr_{\sigma''|\,|\sigma''|-1}
    =2R\bfr_{\sigma|\, |\sigma|-1}
    >2r,
\end{align*}
which implies the claim.

Consequently, any $\sigma\in\Sigma_b(r)$ can be written in the form $\sigma=i\omega\tau$ with $i\leq N$, $\omega\in\Xi^{[i]}_*$ and $\tau\in\Sigma^{[i\omega]}_m$.
The similarity factor $\bfr_\omega$ of $\omega$ lies in the interval $[r/(r_{\max})^m,r/(r_{\min})^m]$. Observe that moving a limited number of coordinates from the end of $\omega$ to the word $\tau$ relaxes the assumptions on $\sigma$. Hence any $\sigma\in\Sigma_b(r)$ can be written in the form $\sigma=i\omega\tau$ with $i\leq N$, $\omega\in\Xi^{[i]}(r_*)$ with $r_*:=r/(r_{\min})^m$, and $\tau\in\Sigma^{[i\omega]}_*$ with $|\tau|\leq m':= m\lceil\ln r_{\min}/\ln r_{\max}\rceil$.
Therefore we have
$$
  r^{D-\delta}\E\,\sharp\left(\Sigma_b(r)\right)
    \leq r^{D-\delta}\E\sum_{i=1}^N\sum_{\omega\in\Xi^{[i]}(r_*)}
    \sum_{\tau\in\Sigma^{[i\omega]}_{m'}}1.
$$
Applying the conditional independence (A2) and the fact that
$$
\E\sharp\Sigma_{m'}
\leq \E \sum_{\tau\in\Sigma_{m'}} (\bfr_\tau/r_{\min}^{m'})^D
= (r_{\min})^{-m'D}\E\sum_{\tau\in\Sigma_{m'}}(\bfr_\tau)^D
=(r_{\min})^{-m'D}=:P,
$$
we infer that
$$    r^{D-\delta}\E\,\sharp\left(\Sigma_b(r)\right)
    \leq P\, r^{D-\delta}\E\sum_{i=1}^N\,\sharp\left(\Xi^{[i]}(r_*)\right).$$
Using (A2) again and \eqref{psi}, we get
\begin{align*}
   r^{D-\delta}\E\,\sharp\left(\Sigma_b(r)\right)
    \leq P\, r^{D-\delta}\E N\, \E\,\sharp\left(\Xi(r_*)\right)
   & = P\, \,r_{\min}^{m(D-\delta)} \E N r_*^{D-\delta} \E\,\sharp\left(\Xi(r_*)\right)\\
   & \leq PQ\, r_{\min}^{m(D-\delta)} \E N,
\end{align*}
which yields the assertion of Lemma~\ref{L_2}.

It remains to verify \eqref{psi}.
Using \eqref{rhoalpha_} and the definition of $\Xi(r)$ we infer for sufficiently large $M$ and $r<\rho$,
\begin{eqnarray*}
\psi(r)&:=&r^{D-\delta}\E\,\sharp\left(\Xi(r)\right)
=r^{D-\delta}\E\sum_{\tau\in\Si(\r)\setminus\Si(\r,\alpha)}\sharp(\Xi^{[\tau]}(r/\bfr_\tau))\\
&=&\E\sum_{\tau\in\Si(\r)\setminus\Si(\r,\alpha)}(\bfr_\tau)^{D-\delta} (r/\bfr_\tau)^{D-\delta}\sharp\big(\Xi^{[\tau]}(r/\bfr_\tau)\big)\\
&=&\sum_{n=1}^M\E\sum_{\substack{\tau\in\Si(\r)\setminus\Si(\r,\alpha)\\ |\tau|=n}} (\bfr_\tau)^{D-\delta}(r/\bfr_\tau)^{D-\delta}\sharp\big(\Xi^{[\tau]}(r/\bfr_\tau)\big)\\
&=&\sum_{n=1}^M\E\sum_{\substack{\tau\in\Si(\r)\setminus\Si(\r,\alpha)\\ |\tau|=n}} (\bfr_\tau)^{D-\delta} \psi(r/\bfr_\tau)
=\E\sum_{\tau\in\Si(\r)\setminus\Si(\r,\alpha)} (\bfr_\tau)^{D-\delta}\psi(r/\bfr_\tau)\\
&\le&\E\sum_{\tau\in\Si(\r)\setminus\Si(\r,\alpha)}(\bfr_\tau)^{D-\delta}\sup_{r'\ge r/\r}\limits\psi(r')=\sup_{r'\ge r/\r}\limits\psi(r')\, ,
\end{eqnarray*}
where we have used \eqref{cond_exp_equation} in the fifth equality and \eqref{delta_} at the end. Hence, $\psi(r)\le\sup_{r'\ge r/\r}\limits\psi(r')$ for any $r<\r$ which implies
$$\sup_{r\ge\r^{k+1}}\psi(r)\le \sup_{r\ge\r^k}\psi(r)~~\mbox{for all}~k\in\N\, .$$
Since the function $\psi$ is bounded on any finite interval away from zero, it is bounded on $(0,R)$. This completes the proof of \eqref{psi} and thus of Lemma~\ref{L_2}.
\end{proof}

Now we are ready to provide the missing proof of Lemma~\ref{lem:main-xi_k}.
\begin{proof}[Proof of Lemma~\ref{lem:main-xi_k}] \label{proof-lem:main-xi_k}
Note that, by definition, $\xi_k(r)=0$ for $r>R$. For $0<r\leq R$, we have
\begin{align*}
    |\xi_k(r)|&=\left| C_k(F(r))-\sum_{i=1}^NC_k(F_i(r))+\sum_{i=1}^N\1_{(Rr_i,R]}(r)C_k(F_i(r))\right|\\
&\leq \left| C_k(F(r))-\sum_{i=1}^NC_k(F_i(r))\right|+\sum_{i=1}^N\1_{(Rr_i,R]}(r)C_k^{\var} (F_i(r))\\
&\leq \left| C_k(F(r))-\sum_{i=1}^NC_k(F_i(r))\right|+b_kNr^k
\end{align*}
for some constant $b_k$, where we have used Lemma~\ref{lems:large distances} in the last estimate, justified by the fact that $|F_i|\leq r_i|O|=(|O|/R)Rr_i$ and $|O|/R>\sqrt{2}$.

Set $A_r:=(\mathfrak{f}(O)^c)(r)$ (see \eqref{fO} for the definition of $\mathfrak{f}$), and note that, since curvature measures are locally defined (and $F(r)\cap O_i=F_i(r)\cap O_i$ for $i\leq N$),

$$C_k(F(r),(A_r)^c)=\sum_{i=1}^NC_k(F_i(r),(A_r)^c).$$
Hence we have
\begin{align*}
    |\xi_k(r)|&\leq \left| C_k(F(r),A_r)-\sum_{i=1}^NC_k((F_i(r),A_r)\right|+b_kNr^k\\
    &=\left| C_k(F(r),A_r)-\sum_{i=1}^N r_i^k C_k(F^{[i]}(r/r_i),f_i^{-1}(A_r))\right|+b_kNr^k\\
    &\leq C_k^{\var}(F(r),A_r)+\sum_{i=1}^Nr_i^kC_k^{\var}(F^{[i]}(r/r_i),f_i^{-1}(A_r))+b_kNr^k\\
    &\leq C_k^{\var}(F(r),A_r)+\sum_{i=1}^Nr_i^kC_k^{\var}(F^{[i]}(r/r_i),A^{[i]}_{r/r_i})+b_kNr^k,
\end{align*}

where we have used the inclusion $f_i^{-1}(A_r)\subset O^c(r/r_i)\subset A^{[i]}_{r/r_i}$ in the last step, cf.\ \cite[p.21]{RWZ23}).

Further, since $F(r)\cap A_r\subset\bigcup_{\sigma\in\Sigma_b(r)}O_\sigma(r)$, we have
\begin{align} \label{eq:Ar-est}
C_k^{\var}(F(r),A_r)=C_k^{\var}(F(r),f(r)\cap A_r)\leq\sum_{\sigma\in\Sigma_b(r)}C_k^{\var}(F(r),O_\sigma(r)).
\end{align}
Using that curvature measures are locally determined, we infer that for any $\sigma\in\Sigma_b(r)$,
\begin{align*}
    &C_k^{\var}(F(r),O_\sigma(r))\\   &=C_k^{\var}\left(F(r),O_\sigma(r)\setminus\bigcup_{\substack{\sigma'\in\Sigma_b(r)\\ \sigma'\neq\sigma}} O_{\sigma'}(r)\right)+C_k^{\var}\left(F(r),O_\sigma(r)\cap\bigcup_{\substack{\sigma'\in\Sigma_b(r)\\\sigma'\neq\sigma}}O_{\sigma'}(r)\right)\\
    &\leq C_k^{\var}(F_\sigma(r))
    +\sum _{\substack{\sigma'\in\Sigma_b(r)\\ \sigma'\neq\sigma}}C_k^{\var}\left(F(r),O_\sigma(r)\cap O_{\sigma'}(r)\right).
\end{align*}
Since $|F_\sigma|\leq r_\sigma |O|<\frac rR\frac{R}{\sqrt{2}}=\frac{r}{\sqrt{2}}$ (by the choice of the constant $R$ in the definition of $\Sigma(r)$, see \eqref{subtree-r}), we have $C_k^{\var}(F_\sigma(r))\leq d_kr^k$ for some constant $d_k>0$ by Lemma~\ref{lems:large distances}. Using further the fact that the number of $\sigma'\in\Sigma(r)$ such that $O_\sigma(r)\cap O_{\sigma'}(r)\neq\emptyset$ is bounded by a constant $\Gamma$ (see \cite[Lemma~5.1]{RWZ23}) and assumption (ii), we get
$$C_k^{\var}(F(r),O_\sigma(r))\leq d_kr^k+\Gamma c_kr^k=:\tilde{c}_kr^k.$$
Combining this with \eqref{eq:Ar-est} yields
\begin{align}\label{eq:Ar-est2} C_k^{\var}(F(r),A_r)\leq\tilde{c}_kr^k\sharp\left(\Sigma_b(r)\right).
\end{align}
Applying this and the corresponding estimates for $C_k^{\var}(F^{[i]}(r/r_i),A^{[i]}_{r/r_i})$, for $i\leq N$, we infer
\begin{align} \notag
  |\xi_k(r)|&\leq \tilde{c}_kr^k\sharp\left(\Sigma_b(r)\right)+\sum_{i=1}^Nr_i^k  \tilde{c}_k\, (r/r_i)^k\sharp\left(\Sigma^{[i]}_b(r/r_i)\right)+b_kNr^k.\\
  & \label{eq:Ar-est3} = \tilde{c}_kr^k\left(\sharp\left(\Sigma_b(r)\right)+\sum_{i=1}^N \sharp\left(\Sigma^{[i]}_b(r/r_i)\right)\right)+b_kNr^k.
\end{align}
Now, in order to show \eqref{locbound}, fix some $r_0>0$, and note that for any $r>r_0$, $\sharp\left(\Sigma_b(r)\right)\leq\sharp\left(\Sigma(r)\right)$ (and similarly $\sharp\left(\Sigma^{[i]}_b(r/r_i)\right)$) is bounded a.s.\ by a constant independent of $r$. (This follows by a simple volume comparison argument, since the family $\{O_\sigma:\, \sigma\in\Sigma(r)\}$ is pairwise disjoint a.s., by the (UOSC) assumption.) Therefore, and since $\E N<\infty$, the estimate \eqref{locbound} follows.

It remains to show \eqref{bound_2}, for which we infer from Lemma~\ref{L_2} that there are  $\delta,d_k>0$ such that
$\E\sharp\left(\Sigma_b(r)\right)\leq d_kr^{\delta-D}$ for all $0<r<R$. For the second summand we apply \eqref{cond_exp_equation_1} with the random function $X(s):=\sharp\Sigma_b(s)$ and $h(\rho,X):=X(r/\rho)$ and obtain
\begin{align*}
\E\sum_{i=1}^N \sharp\left(\Sigma^{[i]}_b(r/r_i)\right)&=\sum_{i=1}^\infty \E\left(\1\{i\in\Sigma_1\} \sharp\left(\Sigma^{[i]}_b(r/r_i)\right)\right)\\
&=\sum_{i=1}^\infty \P (i\in\Sigma_1)\int \E\,\sharp\left(\Sigma_b(r/\rho)\right)\P_{r_i}(d\rho| i\in\Sigma_1),
\end{align*}
where $\P_{r_i}(\cdot| i\in\Sigma_1)$ is the conditional distribution of $r_i$ under the condition $i\in\Sigma_1$.
Applying Lemma~\ref{L_2} again, we infer that the integrand is bounded by the expression ${d}_k  (r/\rho)^{\delta-D}\leq {d}_k r_{\max}^{D-\delta} r^{\delta-D}$ (with the same constants $d_k,\delta>0$ as above) and so the integral is bounded by the same constant (whenever $i\leq N$).  We conclude, that the whole expression is bounded by ${d}_k r_{\max}^{D-\delta} \E N\, r^{\delta-D}$. Now \eqref{bound_2} follows by taking expectations in \eqref{eq:Ar-est3} and using the above estimates. Here we also used that, by Lemma~\ref{L_2}, $D-\delta>0$ such that the terms containing $r^{k-D+\delta}$ dominate the one containing $r^k$.
\end{proof}
{\it Proof of Theorem} \ref{positive}. The argument is literally the same as the one in the proof of \cite[Theorem~2.3]{Za23}. Instead of the independence structure used there for the expectation estimates, here the back path property is employed. Therefore we omit the details here.

\vspace{3mm}
\begin{rem}\label{rem:extension}
An analysis of the proofs shows that our method can be extended to the following more general class of models.
Instead of Assumption (A2) suppose that there exists a random stopping time $\nu=\nu(\mathcal{T})$ in the levels of the tree (with respect to the filtration given by the labels up to the different steps) such that
\begin{align} \notag
\mathbb{P}\big((f_{\sigma_1},&f_{\sigma_1\sigma_2},\ldots, f_{\sigma_1\ldots\sigma_n})\in B_n,\, \sigma\in\Sigma_n,
\,\nu(\mathcal{T})=n,\, \mathcal{T}^{[\sigma]}\in C\big) \\
=\, &\mathbb{P}\big((f_{\sigma_1},f_{\sigma_1\sigma_2},
\ldots,f_{\sigma_1\ldots\sigma_n})\in B_n,\,\sigma\in\Sigma_n,\, \nu(\mathcal{T})=n\big)\, \mathbb{P}(\mathcal{T}\in C)
\end{align}
holds for all measurable sets $B_n$ and $C$ in the corresponding spaces.
Moreover, suppose that the mean number of elements of the Markov stop
$\Sigma^\nu:=\{\sigma\in\Sigma_*: |\sigma|=\nu\}$ is finite.

Then the mean Minkowski dimension $D$ of $F$ in the sense of Remark~\ref{rem:dimensions} is determined by the equation
$$
\mathbb{E}\sum_{\sigma\in\Sigma^\nu} {\bf r}_\sigma^D=1.
$$
Moreover, the assertions of Theorem~\ref{maintheorem} remain valid, if in the definition of the function $R_k$ the summation over $i\leq N$ is replaced by that over the elements $\sigma\in\Sigma^\nu$ and the random sets $F_i$ are replaced by the sets $F_\sigma$, accordingly. In addition,  the set $\Sigma(r)$ in condition (ii) must be redefined. Here and in the proofs
the sets $\Sigma_n$ have to be replaced by the inductively defined code sets $$\Sigma_n^\nu:=\big\{\sigma\tau\in\Sigma_*:\,\sigma\in\Sigma_{n-1}^\nu,\,|\tau|=\nu(\T^{[\sigma]})\big\},$$
where $\Sigma_0^\nu=\emptyset$.
The renewal theorem is applied with respect to the random time $\nu$, the subtrees at this level and the distribution
$$\mu:=\mathbb{E}\sum_{\sigma\in\Sigma_\nu} \bf{1}(|\ln({\bf r}_\sigma)|\in(\cdot))\, {\bf r}_\sigma^D.$$
Note that the above assumption (A2) can be considered as the special case $\nu=1$. In the model of \cite{Tr21}, which goes back to \cite{JJK+14}, the random number $\nu$ coincides with the first neck in the tree.
\end{rem}

\end{document}